\documentclass[a4paper,12pt,leqno]{amsart}
\usepackage{latexsym}
\usepackage[all]{xy}

\usepackage{amssymb} 
\usepackage{amsmath} 
\usepackage{color}
\usepackage{comment}

\definecolor{gray}{gray}{0.7}
\definecolor{Gray}{gray}{0.3}

\usepackage{amscd}

\numberwithin{equation}{section}

\newtheorem{theorem}{Theorem}[section]
\newtheorem{lemma}[theorem]{Lemma} 
\newtheorem{corollary}[theorem]{Corollary}
\newtheorem{proposition}[theorem]{Proposition} 
\newtheorem{conjecture}[theorem]{Conjecture} 
\newtheorem{remark}[theorem]{Remark}
\newtheorem{example}[theorem]{Example}

\allowdisplaybreaks[3]

\def\C{\mathbb C}
\def\Q{\mathbb Q}
\def\Z{\mathbb Z}

\DeclareMathOperator{\Hess}{Hess}

\DeclareMathOperator{\Ind}{Ind}

\newcommand{\Flags}{Flag}
\newcommand{\dimX}{m}
\newcommand{\into}{\hookrightarrow}
\newcommand{\onto}{\twoheadrightarrow}
\def\Sn{\mathfrak{S}_n}
\def\S{\mathfrak{S}}
\def\Tn{T}
\def\Xh{X(h)}
\def\Xhh{X(h')}
\def\Ph{P(X(h),q)}
\def\Phh{P(X(h'),q)}
\def\A{A}
\def\q{q}
\def\x{x}
\def\y{y}
\def\cx{\check{x}}
\def\cy{\check{y}}
\def\cf{\check{f}}

\begin{document}
  
\title[Cohomology rings of Hessenberg varieties]{The cohomology rings of regular semisimple Hessenberg varieties for $h=(h(1),n,\ldots,n)$}
\author {Hiraku Abe}
\address{Osaka City University Advanced Mathematical Institute, 3-3-138 Sugimoto, Sumiyoshi-ku, Osaka 558-8585, 
Japan}
\email{hirakuabe@globe.ocn.ne.jp}

\author {Tatsuya Horiguchi}
\address{Department of Pure and Applied Mathematics,
Graduate School of Information Science and Technology,
Osaka University,
Suita, Osaka, 565-0871, Japan}
\email{tatsuya.horiguchi0103@gmail.com}

\author {Mikiya Masuda}
\address{Department of Mathematics, Osaka City University, 3-3-138 Sugimoto, Sumiyoshi-ku, Osaka 558-8585, Japan}
\email{masuda@sci.osaka-cu.ac.jp}
\date{\today}



\begin{abstract}
We investigate the cohomology rings of regular semisimple Hessenberg varieties whose Hessenberg functions are of the form $h=(h(1),n\dots,n)$ in Lie type $A_{n-1}$. The main result of this paper gives an explicit presentation of the cohomology rings in terms of generators and their relations.
Our presentation naturally specializes to Borel's presentation of the cohomology ring of the flag variety and it is compatible with the representation of the symmetric group $\mathfrak{S}_n$ on the cohomology constructed by J. Tymoczko.
As a corollary, we also give an explicit presentation of the $\mathfrak{S}_n$-invariant subring of the cohomology ring.
\end{abstract}

\maketitle

\setcounter{tocdepth}{1}

\section{Introduction}
Hessenberg varieties form an interesting class of algebraic subsets of the full flag variety, and they have been studied in the contexts of geometry (see for example \cite{ma-pr-sh, br-ca04, ty, InskoYong, precup13a, ab-de-ga-ha}),  combinatorics (see e.g. \cite{so-ty, tymo08, an-ty, ha-ty, dr2015, Horiguchi}), and representation theory (see e.g. \cite{Springer76, Tanisaki, proc90, GarsiaProcesi, mb-ty13, AbeHoriguchi}). 
Let $A:\C^n\rightarrow \C^n$ be a linear operator and $h:\{1,2,\dots,n\}\rightarrow\{1,2,\dots,n\}$ a Hessenberg function, i.e. a non-decreasing function satisfying $h(i)\geq i \ (1\leq i\leq n)$. The \textbf{Hessenberg variety} (in Lie type $A_{n-1}$) associated to $A$ and $h$ is defined as
\[
\Hess(A,h) :=\{V_\bullet \in \Flags(\C^n) \mid AV_i \subset V_{h(i)} \ \mbox{for all} \ 1\leq i\leq n \}
\]
where $V_\bullet=( V_1 \subset V_2 \subset \cdots \subset V_n=\C^n)$ is a sequence of linear subspaces of $\C^n$ with $\dim_{\C}V_i=i \ (1\leq i\leq n)$. 
A celebrated example is the Springer fiber, i.e. taking $A$ to be a nilpotent operator and the Hessenberg function $h$ to be the identity function, and it is well-known that the Springer fiber plays a fundamental role in geometric representation theory of the symmetric group $\Sn$ (\cite{Springer76}).
Geometry and topology of $\Hess(A,h)$ reflects properties of the linear operator $A$, and it appears in several areas of mathematics (e.g. representation theory, graph theory, and theory of hyperplane arrangements) depending on choices of appropriate operators.

When we take a regular semisimple operator $S$ (i.e. a diagonalizable matrix with distinct eigenvalues), the associated Hessenberg variety $\Hess(S,h)$ is called a \textbf{regular semisimple Hessenberg variety}, and it is a smooth complex submanifold of the flag variety for all $h$ (\cite{ma-pr-sh}).
Among them, when we take the Hessenberg function $h$ to be the one satisfying $h(i)=n\ (1\leq i\leq n)$, it is the ambient flag variety itself, and when we take $h(i)=i+1\ (1\leq i<n)$, it is the toric variery associated with the fan consisting of the collection of Weyl chambers of the root system of Lie type $A_{n-1}$ (\cite{ma-pr-sh}). 
Since the latter is essentially the minimum case, it means that the regular semisimple Hessenberg varieties $\Hess(S,h)$ provide us a (discrete) family of complex submanifolds of the flag variety, connecting the flag variety itself and the toric variety described above.
The cohomology rings of the two extremal cases (i.e. the flag variety and the toric variety) are well understood and presented explicitly, but for the intermediate cases the ring structure of the cohomology has not been studied well. In particular, explicit presentations of the cohomology rings of intermediate cases are not known. So it is interesting both from geometric and algebraic point of view to determine the cohomology rings of regular semisimple Hessenberg varieties.
The cohomology rings of some other kinds of Hessenberg varieties have been studied in \cite{Tanisaki, fu-ha-ma, ha-ho-ma, Abe-Crooks, ab-ha-ho-ma, AHMMS} for example, and 
in these cases the generators of the cohomology rings come from the cohomology of the flag variety. However, in the case of regular semisimple Hessenberg varieties, the restriction map from the cohomology of the flag variety is \textit{not} surjective in general, and there is a difficulty of finding those cohomology classes which do not come from the flag variety.
On the other hand, regular semisimple Hessenberg varieties have natural torus actions as in the case of the flag variety, and their equivariant cohomology rings admit a presentation called \textbf{GKM presentation} developped in \cite{Go-Ko-Ma} so that we can study the ordinary cohomology rings  from the torus equivariant cohomology rings (\cite{tymo08}).

Another aspect of regular semisimple Hessenberg varieties is the representation of the symmetric group $\Sn$ on their cohomology $H^*(\Hess(S,h);\C)$ constructed by J. Tymoczko (\cite{tymo08}). 
Shareshian and Wahcs observed its dependence on the Hessenberg function $h$, and they announced a beautiful conjecture which states that the $\Sn$-representation on $H^*(\Hess(S,h);\C)$ is described by the chromatic quasisymmetric function of a graph associated to the Hessenberg function $h$ (\cite{sh-wa11, sh-wa14}). This conjecture has been proved by an algebro-geometric method given by Brosnan-Chow (\cite{br-ch}), and soon later Guay-Paquet gave a combinatorial proof (\cite{gu}). Since Shareshian and Wachs had already described the Schur basis expansion of the chromatic quasisymmetric function in \cite{sh-wa14}, this means that the irreducible decomposition of the $\Sn$-representation on the cohomology is determined for arbitrary $h$.
However, the Shareshian-Wachs conjecture was announced as a major step toward to the Stanley-Stembridge conjecture in graph theory which states that the chromatic quasisymmetric function of the incomparability graph of a $(3+1)$-free poset is e-positive (\cite[Conjecture 5.5]{st-ste}, \cite[Conjecture 5.1]{Stanley1995}).
In fact, to solve the Stanley-Stembridge conjecture, Shareshian and Wahcs also conjectured that the $\Sn$-representation on the cohomology $H^*(\Hess(S,h);\C)$ is a \textit{permutation representation} (\cite[Conjecture 5.4]{sh-wa11}, \cite[Conjecture 10.4]{sh-wa14}). It is known that the Stanley-Stembridge conjecture is true if so is the latter conjecture given by Shareshian-Wahcs (See \cite{sh-wa14} for details).

The $\Sn$-representation on the cohomology $H^*(\Hess(S,h);\C)$ of the regular semisimple Hessenberg variety also plays a role of a connection of $\Hess(S,h)$ to other \textit{regular Hessenberg varieties} (i.e. the ones defined for regular matrices). 
As an evidence, the $\Sn$-invariant subring of $H^*(\Hess(S,h);\C)$ in fact gives us the cohomology ring of the \textit{regular nilpotent Hessenberg variety} with the same function $h$ (\cite{AHHM}).
Also, if $\lambda$ is a partition of $n$, then there is an associated Young subgroup $\mathfrak{S}_{\lambda}\subset\Sn$, and the dimension of the $\mathfrak{S}_{\lambda}$-invariant subgroup of $H^i(\Hess(S,h);\C)$ coincides with the dimension of the degree $i$ component of the cohomology of regular Hessenberg variety of type $\lambda$ with the same function $h$, for all integers $i$ (\cite{br-ch}).

In this paper, we determine the ring structure of the cohomology $H^*(\Hess(S,h);\Z)$ of regular semisimple Hessenberg varieties in $\Z$ coefficients whose Hessenberg functions are of the form $h=(h(1),n,\dots,n)$, i.e. $1\leq h(1)\leq n$ and $h(j)=n \ (2\leq j\leq n)$. In particular, we give a finite list of ring generators which is compatible with Tymoczko's $\Sn$-action, and we determine their relations (Theorem \ref{theorem:maincohomology}). There are three advantages of our presentation; (1) it extends Borel's presentation of the cohomology ring of the flag variety, (2) it manifestly exhibits that the $\Sn$-representation in our case is a permutation representation, (3) the $\Sn$-invariant subring $H^*(\Hess(S,h))^{\Sn}$ is presented as well, and in $\Q$ coefficients it is naturally identified with the presentation of the cohomology ring of the corresponding regular nilpotent Hessenberg variety given in \cite{AHHM}.
We note that, in our case $h=(h(1),n,\dots,n)$, it is known that  the $\Sn$-representation on $H^*(\Hess(S,h);\C)$ is a permutation representation (See Section \ref{subsec:Sn-rep} for details), 
but our result shows that it is visible at the level of its ring structure.

This paper is organized as follows. In Section 2, we briefly recall basic properties of regular semisimple Hessenberg varieties and Tymoczko's $\Sn$-action on their cohomology rings. 
After that, we concentrate on the case $h=(h(1),n,\dots,n)$, and  we describe an arithmetic formula for their Poincar$\acute{\text{e}}$ polynomials in Section 3. 
Finally, we describe the ring structure of cohomology of the regular semisimple Hessenberg varieties in Section 4, and we see that the $\Sn$-action is naturally described in our presentation.

\bigskip
\noindent \textbf{Acknowledgements.}  
The first author is partially supported by a JSPS Grant-in-Aid for Young Scientists (B): 15K17544 and a JSPS Research Fellowship for Young Scientists Postdoctoral Fellow: 16J04761.
The second author was partially supported by JSPS Grantin-Aid
for JSPS Fellows 15J09343. 
The third author is partially supported by JSPS
Grant-in-Aid for Scientific Research (C) 16K05152.

\section{Background and notations}
\label{sec:background}

In this section, we recall some background, and establish some notations for the rest of the paper. 
We work with cohomology with coefficients in $\Z$ throughout this paper unless otherwise specified.

\subsection{Regular semisimple Hessenberg varieties}
\label{subsec:background}
Let $n$ be a positive integer, and we use the notation $[n]:=\{1,2,\ldots,n\}$ throughout this paper.
A {\bf Hessenberg function} is a nondecreasing function $h: [n] \to [n]$ satisfying $h(i) \geq i$ for all $i \in [n]$. 
We frequently write a Hessenberg function by listing its values in a sequence, i.e. $h=(h(1),h(2),\ldots,h(n))$.
We may identify a Hessenberg function $h$ with a configuration of the shaded boxes on a square grid of size $n \times n$ which consists of boxes in the $i$-th row and the $j$-th column satisfying $i \leq h(j)$ for $i,j\in[n]$, as we illustrate in the following example.

\begin{example}\label{example:HessenbergFunction}
Let $n=5$. The Hessenberg function $h=(3,3,4,5,5)$ corresponds to the configuration of the shaded boxes drawn in Figure $\ref{pic:stair-shape}$. \\
\begin{figure}[h]
\begin{center}
\begin{picture}(75,75)
\put(0,63){\colorbox{gray}}
\put(0,67){\colorbox{gray}}
\put(0,72){\colorbox{gray}}
\put(4,63){\colorbox{gray}}
\put(4,67){\colorbox{gray}}
\put(4,72){\colorbox{gray}}
\put(9,63){\colorbox{gray}}
\put(9,67){\colorbox{gray}}
\put(9,72){\colorbox{gray}}

\put(15,63){\colorbox{gray}}
\put(15,67){\colorbox{gray}}
\put(15,72){\colorbox{gray}}
\put(19,63){\colorbox{gray}}
\put(19,67){\colorbox{gray}}
\put(19,72){\colorbox{gray}}
\put(24,63){\colorbox{gray}}
\put(24,67){\colorbox{gray}}
\put(24,72){\colorbox{gray}}

\put(30,63){\colorbox{gray}}
\put(30,67){\colorbox{gray}}
\put(30,72){\colorbox{gray}}
\put(34,63){\colorbox{gray}}
\put(34,67){\colorbox{gray}}
\put(34,72){\colorbox{gray}}
\put(39,63){\colorbox{gray}}
\put(39,67){\colorbox{gray}}
\put(39,72){\colorbox{gray}}

\put(45,63){\colorbox{gray}}
\put(45,67){\colorbox{gray}}
\put(45,72){\colorbox{gray}}
\put(49,63){\colorbox{gray}}
\put(49,67){\colorbox{gray}}
\put(49,72){\colorbox{gray}}
\put(54,63){\colorbox{gray}}
\put(54,67){\colorbox{gray}}
\put(54,72){\colorbox{gray}}

\put(60,63){\colorbox{gray}}
\put(60,67){\colorbox{gray}}
\put(60,72){\colorbox{gray}}
\put(64,63){\colorbox{gray}}
\put(64,67){\colorbox{gray}}
\put(64,72){\colorbox{gray}}
\put(69,63){\colorbox{gray}}
\put(69,67){\colorbox{gray}}
\put(69,72){\colorbox{gray}}

\put(0,48){\colorbox{gray}}
\put(0,52){\colorbox{gray}}
\put(0,57){\colorbox{gray}}
\put(4,48){\colorbox{gray}}
\put(4,52){\colorbox{gray}}
\put(4,57){\colorbox{gray}}
\put(9,48){\colorbox{gray}}
\put(9,52){\colorbox{gray}}
\put(9,57){\colorbox{gray}}

\put(15,48){\colorbox{gray}}
\put(15,52){\colorbox{gray}}
\put(15,57){\colorbox{gray}}
\put(19,48){\colorbox{gray}}
\put(19,52){\colorbox{gray}}
\put(19,57){\colorbox{gray}}
\put(24,48){\colorbox{gray}}
\put(24,52){\colorbox{gray}}
\put(24,57){\colorbox{gray}}

\put(30,48){\colorbox{gray}}
\put(30,52){\colorbox{gray}}
\put(30,57){\colorbox{gray}}
\put(34,48){\colorbox{gray}}
\put(34,52){\colorbox{gray}}
\put(34,57){\colorbox{gray}}
\put(39,48){\colorbox{gray}}
\put(39,52){\colorbox{gray}}
\put(39,57){\colorbox{gray}}

\put(45,48){\colorbox{gray}}
\put(45,52){\colorbox{gray}}
\put(45,57){\colorbox{gray}}
\put(49,48){\colorbox{gray}}
\put(49,52){\colorbox{gray}}
\put(49,57){\colorbox{gray}}
\put(54,48){\colorbox{gray}}
\put(54,52){\colorbox{gray}}
\put(54,57){\colorbox{gray}}

\put(60,48){\colorbox{gray}}
\put(60,52){\colorbox{gray}}
\put(60,57){\colorbox{gray}}
\put(64,48){\colorbox{gray}}
\put(64,52){\colorbox{gray}}
\put(64,57){\colorbox{gray}}
\put(69,48){\colorbox{gray}}
\put(69,52){\colorbox{gray}}
\put(69,57){\colorbox{gray}}

\put(0,33){\colorbox{gray}}
\put(0,37){\colorbox{gray}}
\put(0,42){\colorbox{gray}}
\put(4,33){\colorbox{gray}}
\put(4,37){\colorbox{gray}}
\put(4,42){\colorbox{gray}}
\put(9,33){\colorbox{gray}}
\put(9,37){\colorbox{gray}}
\put(9,42){\colorbox{gray}}

\put(15,33){\colorbox{gray}}
\put(15,37){\colorbox{gray}}
\put(15,42){\colorbox{gray}}
\put(19,33){\colorbox{gray}}
\put(19,37){\colorbox{gray}}
\put(19,42){\colorbox{gray}}
\put(24,33){\colorbox{gray}}
\put(24,37){\colorbox{gray}}
\put(24,42){\colorbox{gray}}

\put(30,33){\colorbox{gray}}
\put(30,37){\colorbox{gray}}
\put(30,42){\colorbox{gray}}
\put(34,33){\colorbox{gray}}
\put(34,37){\colorbox{gray}}
\put(34,42){\colorbox{gray}}
\put(39,33){\colorbox{gray}}
\put(39,37){\colorbox{gray}}
\put(39,42){\colorbox{gray}}

\put(45,33){\colorbox{gray}}
\put(45,37){\colorbox{gray}}
\put(45,42){\colorbox{gray}}
\put(49,33){\colorbox{gray}}
\put(49,37){\colorbox{gray}}
\put(49,42){\colorbox{gray}}
\put(54,33){\colorbox{gray}}
\put(54,37){\colorbox{gray}}
\put(54,42){\colorbox{gray}}

\put(60,33){\colorbox{gray}}
\put(60,37){\colorbox{gray}}
\put(60,42){\colorbox{gray}}
\put(64,33){\colorbox{gray}}
\put(64,37){\colorbox{gray}}
\put(64,42){\colorbox{gray}}
\put(69,33){\colorbox{gray}}
\put(69,37){\colorbox{gray}}
\put(69,42){\colorbox{gray}}

%
%
\put(30,18){\colorbox{gray}}
\put(30,22){\colorbox{gray}}
\put(30,27){\colorbox{gray}}
\put(34,18){\colorbox{gray}}
\put(34,22){\colorbox{gray}}
\put(34,27){\colorbox{gray}}
\put(39,18){\colorbox{gray}}
\put(39,22){\colorbox{gray}}
\put(39,27){\colorbox{gray}}

\put(45,18){\colorbox{gray}}
\put(45,22){\colorbox{gray}}
\put(45,27){\colorbox{gray}}
\put(49,18){\colorbox{gray}}
\put(49,22){\colorbox{gray}}
\put(49,27){\colorbox{gray}}
\put(54,18){\colorbox{gray}}
\put(54,22){\colorbox{gray}}
\put(54,27){\colorbox{gray}}

\put(60,18){\colorbox{gray}}
\put(60,22){\colorbox{gray}}
\put(60,27){\colorbox{gray}}
\put(64,18){\colorbox{gray}}
\put(64,22){\colorbox{gray}}
\put(64,27){\colorbox{gray}}
\put(69,18){\colorbox{gray}}
\put(69,22){\colorbox{gray}}
\put(69,27){\colorbox{gray}}

%
%
%
\put(45,3){\colorbox{gray}}
\put(45,7){\colorbox{gray}}
\put(45,12){\colorbox{gray}}
\put(49,3){\colorbox{gray}}
\put(49,7){\colorbox{gray}}
\put(49,12){\colorbox{gray}}
\put(54,3){\colorbox{gray}}
\put(54,7){\colorbox{gray}}
\put(54,12){\colorbox{gray}}

\put(60,3){\colorbox{gray}}
\put(60,7){\colorbox{gray}}
\put(60,12){\colorbox{gray}}
\put(64,3){\colorbox{gray}}
\put(64,7){\colorbox{gray}}
\put(64,12){\colorbox{gray}}
\put(69,3){\colorbox{gray}}
\put(69,7){\colorbox{gray}}
\put(69,12){\colorbox{gray}}

\put(0,0){\framebox(15,15)}
\put(15,0){\framebox(15,15)}
\put(30,0){\framebox(15,15)}
\put(45,0){\framebox(15,15)}
\put(60,0){\framebox(15,15)}
\put(0,15){\framebox(15,15)}
\put(15,15){\framebox(15,15)}
\put(30,15){\framebox(15,15)}
\put(45,15){\framebox(15,15)}
\put(60,15){\framebox(15,15)}
\put(0,30){\framebox(15,15)}
\put(15,30){\framebox(15,15)}
\put(30,30){\framebox(15,15)}
\put(45,30){\framebox(15,15)}
\put(60,30){\framebox(15,15)}
\put(0,45){\framebox(15,15)}
\put(15,45){\framebox(15,15)}
\put(30,45){\framebox(15,15)}
\put(45,45){\framebox(15,15)}
\put(60,45){\framebox(15,15)}
\put(0,60){\framebox(15,15)}
\put(15,60){\framebox(15,15)}
\put(30,60){\framebox(15,15)}
\put(45,60){\framebox(15,15)}
\put(60,60){\framebox(15,15)}
\end{picture}
\end{center}
\caption{The configuration corresponding to $h=(3,3,4,5,5)$.}
\label{pic:stair-shape}
\end{figure}
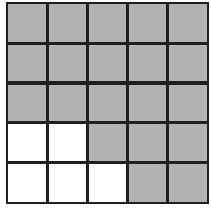
\end{example}

The flag variety $\Flags(\C^n)$ consists of nested sequences of linear subspaces of $\C^n$
$$
V_\bullet:=( V_1 \subset V_2 \subset \cdots \subset V_n=\C^n)
$$
where $V_i$ is of dimension $i$.
For a linear operator $\A: \C^n \to \C^n$ and a Hessenberg function $h: [n] \to [n]$, a {\bf Hessenberg variety} is defined as follows:
$$
\Hess(A,h)=\{V_\bullet \in \Flags(\C^n) \mid AV_i \subset V_{h(i)} \ \mbox{for all} \ i\in[n] \}.
$$
In this paper, we focus on $\Hess(S,h)$ where $S$ is a regular semisimple operator (i.e. a diagonalizable matrix with distinct eigenvalues). The $\Hess(S,h)$ is called a {\bf regular semisimple Hessenberg variety}. In what follows, we assume that $S$ is diagonal with respect to the standard basis of $\C^n$.
Since its topology does not depend on the eigenvalues of $S$ (see specifically Theorem \ref{theorem:Xh} and Proposition~\ref{proposition:GKM} below), we denote $\Hess(S,h)$ simply by $\Xh$. We note that when $h(i)=n \ (1\leq i\leq n)$ we have $\Xh=\Flags(\C^n)$ and when $h(i)=i+1\ (1\leq i<n)$ it is known that $\Xh$ is the toric variety associated with the fan consisting of the collection of Weyl chambers of the root system of type $A_{n-1}$ (\cite{ma-pr-sh}). The following theorem gives some basic properties of $\Xh$.

\begin{theorem} $($\cite[Theorem~6 and Theorem~8]{ma-pr-sh} \label{theorem:Xh}$)$
Let $\Xh$ be a regular semisimple Hessenberg variety. Then the following hold. 
\begin{itemize}
\item[(1)] $\Xh$ is smooth equidimensional of dimension equal to $\displaystyle\sum _{i=1}^n (h(i)-i)$.
\item[(2)] The integral cohomology group of $X$ is torsion-free, and the odd-degree cohomology groups vanish. 
\item[(3)] Let $\Ph$ denote the Poincar\'{e} polynomial of $\Xh$ with $\deg(q)=2$.
Then we have
$$
\Ph=\displaystyle\sum _{w\in \Sn} \q^{d_h(w)}
$$
where $\Sn$ is the $n$-th symmetric group and $d_h(w)$ is the cardinality of the following set:
$$
D_h(w):=\{(j,i) \mid 1\leq j < i \leq n, \ w(j)>w(i), \ i\leq h(j) \}. 
$$ 
\end{itemize}
\end{theorem}

\smallskip

\subsection{Torus equivariant cohomology of $\Xh$}
Let $\Tn$ be the following $n$-dimensional torus in the general linear group ${\rm GL}(n, \C)$:
$$
\Tn:=\left\{ \begin{pmatrix} g_1 &      &          &  \\ 
                                           & g_2 &           &  \\
                                           &      & \ddots  &  \\
                                           &      &            & g_n 
\end{pmatrix} \middle| \ g_i\in \C^* \ {\rm for \ all } \ i=1,2,\ldots,n \right\}.
$$ 
Since the flag variety $\Flags(\C^n)$ has a natural action of ${\rm GL}(n, \C)$, the torus $\Tn$ acts on $\Flags(\C^n)$ by its restriction. This $\Tn$-action preserves a regular semisimple Hessenberg variety $\Xh$ since the elements of $\Tn$ commute with the diagonal operator $S$, and $\Xh$ contains all the $\Tn$-fixed points $\Flags(\C^n)^{\Tn}$ of the flag variety $\Flags(\C^n)$ (\cite[Proposition~3]{ma-pr-sh}).
So we see that 
$$
\Xh^{\Tn}=\Flags(\C^n)^{\Tn}\cong \Sn
$$
where $\Sn$ is the $n$-th symmetric group.
Here, the identification $\Flags(\C^n)^{\Tn}\cong \Sn$ is given by sending $w\in \Sn$ to the flag $V_{\bullet}$ specified by $V_i:={ \rm span }_{\C}\{e_{w(1)},\ldots,e_{w(i)}\}$ where $\{e_{1},\ldots,e_{n}\}$ is the standard basis of $\C^n$.

The $\Tn$-equivariant cohomology $H^*_{\Tn}(\Xh)$ is defined to be the ordinary cohomology $H^*(E{\Tn} \times_{\Tn} \Xh)$ where $E{\Tn} \to B{\Tn}$ is a universal principal bundle of $\Tn$. In particular, we have ring homomorphisms 
$$
H^*(B{\Tn}) \to H^*_{\Tn}(\Xh) \to H^*(\Xh)
$$
since $E{\Tn} \times_{\Tn} \Xh \to B{\Tn}$ is a fiber bundle with fiber $\Xh$.
Let $t_i$ be the first Chern class of the line bundle over $BT$ corresponding to the projection $\Tn \to \C^*$; $diag(g_1,\ldots,g_n) \mapsto g_i$.
Then we may identify $H^*(B{\Tn})$ with the polynomial ring $\Z[t_1,\dots,t_n]$ with $\deg(t_i)=2$ for all $i=1,\ldots,n$. Since $\Xh$ has no odd-degree cohomology as we saw in Theorem~\ref{theorem:Xh} (2), we have the following properties (\cite{mi-to}):
\begin{itemize}
 \item $\Z[t_1,\dots,t_n]=H^*(B{\Tn}) \to H^*_{\Tn}(\Xh)$ is injective,
 \item $H^*_{\Tn}(\Xh) \to H^*(\Xh)$ is surjective,
 \item $H^*(\Xh) \cong H^*_{\Tn}(\Xh)/(t_1,\dots,t_n)$ where $(t_1,\dots,t_n)$ is the ideal of $H^*_{\Tn}(\Xh)$ generated by $t_1,\dots,t_n$,
 \item the restriction map $H^*_{\Tn}(\Xh) \to H^*_{\Tn}(\Xh^{\Tn})=\bigoplus _{w \in \Sn} \Z[t_1,\dots,t_n]$ is injective. 
\end{itemize}
The second item above is particularly important for us; it means that we can study the ordinary cohomology ring $H^*(\Xh)$ from the equivariant cohomology ring $H^*_{\Tn}(\Xh)$, and the last item enables us to identify $H^*_{\Tn}(\Xh)$ with a subring of $\bigoplus _{w \in \Sn} \Z[t_1,\dots,t_n]$ which is described in the following proposition.

\begin{proposition} $($\cite[Proposition~4.7]{tymo08} \label{proposition:GKM}$)$
The image of the restriction map above $H^*_{\Tn}(\Xh) \into \bigoplus _{w \in \Sn} \Z[t_1,\dots,t_n]$ is given by 
\begin{equation}\label{eq:GKM}
\left\{ \alpha \in \bigoplus_{w\in\Sn} \Z[t_1,\dots,t_n] \left|
\begin{matrix} \text{ $\alpha(w)-\alpha(w')$ is divisible by $t_{w(i)}-t_{w(j)}$} \\
\text{ if there exist $1\leq j<i\leq n$ satisfying} \\
 \text{ $w'=w(j \ i)$ and $i \leq h(j)$} \end{matrix} \right.\right\}
\end{equation}
where $\alpha(w)$ is the $w$-component of $\alpha$ and $(j \ i)$ is the transposition of $j$ and $i$. 
\end{proposition}

\begin{remark}
Proposition~$4.7$ in \cite{tymo08} is given in $\C$-coefficients, but it is valid for $\Z$-coefficients as well because of the following two reasons; $(1)$ $\Xh$ has a cellular decomposition whose cells consist of the intersections of $\Xh$ and the Schubert cells, namely a Bialynicki-Birula's cell  decomposition $($see \cite{ma-pr-sh} for details$)$, $(2)$ the weights of the tangential representation of $T$ at each fixed point are primitive vectors in $H^2(B\Tn)$, and they are pairwise relatively prime in $H^*(B\Tn)$ $($cf. \cite[Theorem~3.1]{ha-he-ho}$)$. In fact, the condition $(1)$ means that the Poincar\'{e} duals of the closures of the cells form an $H^*(B\Tn)$-basis of $H_{\Tn}^*(\Xh)$, and the condition $(2)$ ensures that the image of the restriction map $H^*_{\Tn}(\Xh) \into \bigoplus _{w \in \Sn} \Z[t_1,\dots,t_n]$ lies on the subring  \eqref{eq:GKM} and that the images of the Poincar\'{e} duals form an $\Z[t_1,\dots,t_n]$-basis of \eqref{eq:GKM}.
\end{remark}

The set in \eqref{eq:GKM} can be described by the so-called GKM graph whose vertex set $\mathcal{V}$ and edge set $\mathcal{E}$ are given as follows:
\begin{align*}
&\mathcal{V}:=\Sn,\\
&\mathcal{E}:=\{ (w,w')\in \Sn \times \Sn \mid w'=w(j \ i) \ {\rm and} \ j<i \leq h(j) \ {\rm for \ some} \ j, i \}.
\end{align*}
Additionally, we equip each edge $(w,w')$ such that $w'=w(j \ i)$ and $j<i \leq h(j)$ with the data of the polynomial $\pm(t_{w(i)}-t_{w(j)})$ (up to sign) arising in \eqref{eq:GKM}.
We call this labeled graph \textbf{the GKM graph of $X(h)$}, and we denote it by $\Gamma(h)$. In this language, the condition in \eqref{eq:GKM} says that the collection of polynomials $(\alpha(w))_{w\in \Sn}$ satisfies the following: if $w$ and $w'$ are connected in $\Gamma(h)$ by an edge labeled by $t_{w(i)}-t_{w(j)}$, then the difference of the polynomials assigned for $w$ and $w'$ must be divisible by the label.
In this paper, we call the condition described in \eqref{eq:GKM} the GKM condition for $\Gamma(h)$.

\begin{example} \label{example:GKM graph}
Let $n=3$. 
For $h=(2,3,3)$ and $h'=(3,3,3)$, the corresponding GKM graphs of $X(h)$ are depicted in Figure $\ref{pic:GKM graphs}$ where we use the one-line notation for each vertex.\\
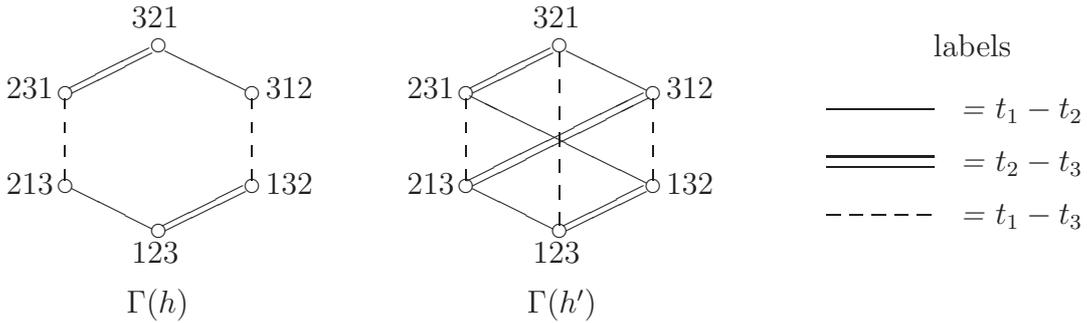
\begin{figure}[h]
\begin{center}
\begin{picture}(400,90)
        \put(50,14){\circle{5}}
        \put(50,84){\circle{5}}
        \put(85,31){\circle{5}}
        \put(85,66){\circle{5}}
        \put(15,31){\circle{5}}
        \put(15,66){\circle{5}}

        \put(47.5,15){\line(-2,1){30}}
        \put(82.5,67){\line(-2,1){30}}
        \put(51.5,16){\line(2,1){30}}
        \put(52.5,14){\line(2,1){30}}
        \put(16.5,68){\line(2,1){30}}
        \put(17.5,66){\line(2,1){30}}
        \put(15,33){\line(0,1){4}}
        \put(15,42){\line(0,1){4}}
        \put(15,51){\line(0,1){4}}
        \put(15,60){\line(0,1){4}}
        \put(85,33){\line(0,1){4}}
        \put(85,42){\line(0,1){4}}
        \put(85,51){\line(0,1){4}}
        \put(85,60){\line(0,1){4}}

        \put(40,2){$123$}
        \put(40,91){$321$}
        \put(90,28){$132$}
        \put(90,64){$312$}
        \put(-7,28){$213$}
        \put(-7,64){$231$}

        \put(38,-17){$\Gamma(h)$}

        \put(200,14){\circle{5}}
        \put(200,84){\circle{5}}
        \put(235,31){\circle{5}}
        \put(235,66){\circle{5}}
        \put(165,31){\circle{5}}
        \put(165,66){\circle{5}}

        \put(232.5,32){\line(-2,1){66}}
        \put(197.5,15){\line(-2,1){30}}
        \put(232.5,67){\line(-2,1){30}}
        \put(165.5,33){\line(2,1){66}}
        \put(167.5,31){\line(2,1){66}}
        \put(201.5,16){\line(2,1){30}}
        \put(202.5,14){\line(2,1){30}}
        \put(166.5,68){\line(2,1){30}}
        \put(167.5,66){\line(2,1){30}}
        \put(200,16){\line(0,1){5}}
        \put(200,26){\line(0,1){5}}
        \put(200,36){\line(0,1){5}}
        \put(200,46){\line(0,1){5}}
        \put(200,56){\line(0,1){5}}
        \put(200,66){\line(0,1){5}}
        \put(200,76){\line(0,1){5}}
        \put(165,33){\line(0,1){4}}
        \put(165,42){\line(0,1){4}}
        \put(165,51){\line(0,1){4}}
        \put(165,60){\line(0,1){4}}
        \put(235,33){\line(0,1){4}}
        \put(235,42){\line(0,1){4}}
        \put(235,51){\line(0,1){4}}
        \put(235,60){\line(0,1){4}}

        \put(190,2){$123$}
        \put(190,91){$321$}
        \put(240,28){$132$}
        \put(240,64){$312$}
        \put(143,28){$213$}
        \put(143,64){$231$}

        \put(188,-17){$\Gamma(h')$}

\put(340,80){{\rm labels}}        
\put(300,60){\line(1,0){40}}
\put(350,56.5){= $t_1-t_2$}
\put(300,42){\line(1,0){40}}
\put(300,38){\line(1,0){40}}
\put(350,36.5){= $t_2-t_3$}
\put(300,20){\line(1,0){5}}
\put(308.5,20){\line(1,0){5}}
\put(317,20){\line(1,0){5}}
\put(325.5,20){\line(1,0){5}}
\put(334,20){\line(1,0){5}}
\put(350,16.5){= $t_1-t_3$}
\end{picture}
\end{center}
\vspace{10pt}
\caption{The GKM graphs $\Gamma(h)$ and $\Gamma(h')$}
\label{pic:GKM graphs}
\end{figure}

\noindent
For example, the collection of polynomials in Figure $\ref{pic:element}$ is an element of $H^*_{\Tn}(X(h))$ but not of $H^*_{\Tn}(X(h'))$. \\
\begin{figure}[h]
\begin{center}
\begin{picture}(100,90)
        \put(50,14){\circle{5}}
        \put(50,84){\circle{5}}
        \put(85,31){\circle{5}}
        \put(85,66){\circle{5}}
        \put(15,31){\circle{5}}
        \put(15,66){\circle{5}}

        \put(47.5,15){\line(-2,1){30}}
        \put(82.5,67){\line(-2,1){30}}
        \put(51.5,16){\line(2,1){30}}
        \put(52.5,14){\line(2,1){30}}
        \put(16.5,68){\line(2,1){30}}
        \put(17.5,66){\line(2,1){30}}
        \put(15,33){\line(0,1){4}}
        \put(15,42){\line(0,1){4}}
        \put(15,51){\line(0,1){4}}
        \put(15,60){\line(0,1){4}}
        \put(85,33){\line(0,1){4}}
        \put(85,42){\line(0,1){4}}
        \put(85,51){\line(0,1){4}}
        \put(85,60){\line(0,1){4}}

        \put(35,2){$t_2-t_3$}
        \put(46,91){$0$}
        \put(90,28){$0$}
        \put(90,64){$0$}
        \put(-22,28){$t_1-t_3$}
        \put(3,64){$0$}

\end{picture}
\end{center}
\vspace{-10pt}
\caption{An element of $H^*_{\Tn}(X(h))$.}
\label{pic:element}
\end{figure}
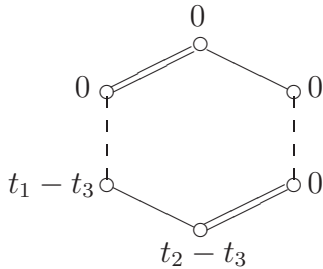
\end{example}

As we pointed out in the Introduction, the restriction map $H^*(\Flags(\C^n))\rightarrow H^*(\Xh)$ is not surjective in general. In fact, whenever $\Xh$ is not $\Flags(\C^n)$, the map is not surjective since the total dimensions of $H^*(\Flags(\C^n))$ and $H^*(\Xh)$ are always the same (Theorem \ref{theorem:Xh} (3)) while the degrees of the highest components are different (Theorem \ref{theorem:Xh} (1)).
However, since we have the graphical presentation of the equivariant cohomology ring $H^*_{\Tn}(\Xh)$ given in Proposition \ref{proposition:GKM}, the surjecitivity of the homomorphism $H^*_{\Tn}(\Xh)\rightarrow H^*(\Xh)$ implies that 
we can use this graphical presentation of $H^*_{\Tn}(\Xh)$ to seek those classes as we will do in Section 4.

\subsection{$\Sn$-action on $H^*_{\Tn}(\Xh)$}
We now describe the  $\Sn$-action on $H^*_{\Tn}(\Xh)$ constructed by Tymoczko \cite{tymo08}. 
For $v \in \Sn$ and $\alpha = (\alpha(w))\in \bigoplus _{w \in \Sn} \Z[t_1,\dots,t_n]$, we define the element $v \cdot \alpha$ by the formula 
\begin{equation}\label{eq:Tymoczko} 
(v\cdot \alpha)(w) := v \cdot \alpha(v^{-1}w) \ \textup{ for all } w \in \Sn
\end{equation}
where $v \cdot f(t_{1},\ldots,t_{n})=f(t_{v(1)},\ldots,t_{v(n)})$ for $f(t_{1},\ldots,t_{n}) \in \Z[t_1,\dots,t_n]$ in the right-hand side. This $\Sn$-action preserves the subset \eqref{eq:GKM}, and hence it defines an $\Sn$-action on the equivariant cohomology $H^*_{\Tn}(\Xh)$ by Proposition~\ref{proposition:GKM}. 
Since we have $v\cdot t_i=t_{v(i)}$ for 
$t_i=(t_i)_{w \in \Sn}$,
the $\Sn$-action on $H^*_{\Tn}(\Xh)$ induces an $\Sn$-action on the ordinary cohomology $H^*(\Xh) \cong H^*_{\Tn}(\Xh)/(t_1,\dots,t_n)$. 

This $\Sn$-action on $H^*(\Xh)$ has a surprising connection with graph theory as follows. 
Shareshian and Wachs formulated a beautiful conjecture relating the chromatic quasisymmetric function of the incomparability graph of a natural unit interval order and Tymoczko's $\Sn$-representation above on the cohomology\footnote[1]{They worked with cohomology with coefficients in $\C$.} of the associated regular semisimple Hessenberg variety (\cite[Conjecture 1.2]{sh-wa11}). 
Also, they determined the coefficients in the Schur basis expansion of the chromatic quasisymmetric function in terms of combinatorics (\cite[Theroem 6.3]{sh-wa14}). 
Recently, the Shareshian-Wachs conjecture was solved by Brosnan and Chow from a geometric viewpoint (\cite{br-ch}), and a combinatorial proof was also given by Guay-Paquet soon later (\cite{gu}). This means that the irreducible decomposition of Tymoczko's $\Sn$-representation on the cohomology $H^*(\Xh;\C)$ is determined for any Hessenberg function $h$.
Interestingly, Shareshian and Wachs conjecture more; they have a further conjecture that the $\Sn$-representation on $H^*(\Xh;\C)$ is a \textit{permutation representation} (see Conjecture \ref{conjecture:h-posi} below), and it is known that if this is true then so is the Stanley-Stembridge conjecture in graph theory which states that the chromatic quasisymmetric function of the incomparability graph of a $(3 + 1)$-free poset is e-positive. That is, the Stanley-Stembridge conjecture is reduced to the following conjecture.

\begin{conjecture}\emph{(}\cite[Conjecture 5.4]{sh-wa11}, \cite[Conjecture 10.4]{sh-wa14}\emph{)}\label{conjecture:h-posi}
The $\Sn$-representation on $H^{2k}(\Xh;\C)$ constructed by Tymoczko is a permutation representation, i.e. a direct sum of  induced representations of the trivial representation from $\S_{\lambda}$ to $\Sn$ where $\S_{\lambda}=\S_{\lambda_1}\times \dots \times \S_{\lambda_{\ell}}$ for $\lambda=(\lambda_1,\ldots,\lambda_{\ell})$.
\end{conjecture}

\bigskip

\section{Poincar\'{e} polynomial of $\Xh$ for $h=(h(1),n,\ldots,n)$}
We described the Poincar\'{e} polynomial of a regular semisimple Hessenberg variety $\Xh$ in Theorem~\ref{theorem:Xh}~(3) as 
$$
\Ph=\displaystyle\sum _{w\in \Sn} \q^{d_h(w)}
$$
where we have $\deg(q)=2$.  
In this section, we give an arithmetic formula for the Poincar\'{e} polynomial $\Ph$ which works for the case $h=(h(1),n,\ldots,n)$, i.e. $1\leq h(1)\leq n$ and $h(j)=n \ (2\leq j\leq n)$.

Recall that we may identify a Hessenberg function with a configuration of the shaded boxes (see Example~\ref{example:HessenbergFunction}). 
We first observe how the Poincar\'{e} polynomial $\Ph$ changes when we add a shaded box to (the configuration of) $h$ so that the resulting configuration corresponds to a Hessenberg function $h'$. We note that the geometric interpretation of this situation is the inclusion $X(h)\subset X(h')$ of codimension $1$ by Theorem \ref{theorem:Xh} (1).

More precisely, let $h$ be a Hessenberg function satisfying the condition $h(r)<h(r+1)$ for some $1 \leq r<n$.
Then, the function defined by
$$
h':=(h(1),\ldots,h(r-1),h(r)+1,h(r+1),\ldots,h(n))
$$ 
is also a Hessenberg function. 
In the language of configurations of shaded boxes, this Hessenberg function $h'$ is obtained from $h$ by placing a new box on the position located at $(h(r)+1)$-th row and the $r$-th column. 
\begin{figure}[h]
\begin{center}
\begin{picture}(200,100)
\put(0,63){\colorbox{gray}}
\put(0,67){\colorbox{gray}}
\put(0,72){\colorbox{gray}}
\put(4,63){\colorbox{gray}}
\put(4,67){\colorbox{gray}}
\put(4,72){\colorbox{gray}}
\put(9,63){\colorbox{gray}}
\put(9,67){\colorbox{gray}}
\put(9,72){\colorbox{gray}}

\put(15,63){\colorbox{gray}}
\put(15,67){\colorbox{gray}}
\put(15,72){\colorbox{gray}}
\put(19,63){\colorbox{gray}}
\put(19,67){\colorbox{gray}}
\put(19,72){\colorbox{gray}}
\put(24,63){\colorbox{gray}}
\put(24,67){\colorbox{gray}}
\put(24,72){\colorbox{gray}}

\put(30,63){\colorbox{gray}}
\put(30,67){\colorbox{gray}}
\put(30,72){\colorbox{gray}}
\put(34,63){\colorbox{gray}}
\put(34,67){\colorbox{gray}}
\put(34,72){\colorbox{gray}}
\put(39,63){\colorbox{gray}}
\put(39,67){\colorbox{gray}}
\put(39,72){\colorbox{gray}}

\put(45,63){\colorbox{gray}}
\put(45,67){\colorbox{gray}}
\put(45,72){\colorbox{gray}}
\put(49,63){\colorbox{gray}}
\put(49,67){\colorbox{gray}}
\put(49,72){\colorbox{gray}}
\put(54,63){\colorbox{gray}}
\put(54,67){\colorbox{gray}}
\put(54,72){\colorbox{gray}}

\put(60,63){\colorbox{gray}}
\put(60,67){\colorbox{gray}}
\put(60,72){\colorbox{gray}}
\put(64,63){\colorbox{gray}}
\put(64,67){\colorbox{gray}}
\put(64,72){\colorbox{gray}}
\put(69,63){\colorbox{gray}}
\put(69,67){\colorbox{gray}}
\put(69,72){\colorbox{gray}}

\put(0,48){\colorbox{gray}}
\put(0,52){\colorbox{gray}}
\put(0,57){\colorbox{gray}}
\put(4,48){\colorbox{gray}}
\put(4,52){\colorbox{gray}}
\put(4,57){\colorbox{gray}}
\put(9,48){\colorbox{gray}}
\put(9,52){\colorbox{gray}}
\put(9,57){\colorbox{gray}}

\put(15,48){\colorbox{gray}}
\put(15,52){\colorbox{gray}}
\put(15,57){\colorbox{gray}}
\put(19,48){\colorbox{gray}}
\put(19,52){\colorbox{gray}}
\put(19,57){\colorbox{gray}}
\put(24,48){\colorbox{gray}}
\put(24,52){\colorbox{gray}}
\put(24,57){\colorbox{gray}}

\put(30,48){\colorbox{gray}}
\put(30,52){\colorbox{gray}}
\put(30,57){\colorbox{gray}}
\put(34,48){\colorbox{gray}}
\put(34,52){\colorbox{gray}}
\put(34,57){\colorbox{gray}}
\put(39,48){\colorbox{gray}}
\put(39,52){\colorbox{gray}}
\put(39,57){\colorbox{gray}}

\put(45,48){\colorbox{gray}}
\put(45,52){\colorbox{gray}}
\put(45,57){\colorbox{gray}}
\put(49,48){\colorbox{gray}}
\put(49,52){\colorbox{gray}}
\put(49,57){\colorbox{gray}}
\put(54,48){\colorbox{gray}}
\put(54,52){\colorbox{gray}}
\put(54,57){\colorbox{gray}}

\put(60,48){\colorbox{gray}}
\put(60,52){\colorbox{gray}}
\put(60,57){\colorbox{gray}}
\put(64,48){\colorbox{gray}}
\put(64,52){\colorbox{gray}}
\put(64,57){\colorbox{gray}}
\put(69,48){\colorbox{gray}}
\put(69,52){\colorbox{gray}}
\put(69,57){\colorbox{gray}}

\put(0,33){\colorbox{gray}}
\put(0,37){\colorbox{gray}}
\put(0,42){\colorbox{gray}}
\put(4,33){\colorbox{gray}}
\put(4,37){\colorbox{gray}}
\put(4,42){\colorbox{gray}}
\put(9,33){\colorbox{gray}}
\put(9,37){\colorbox{gray}}
\put(9,42){\colorbox{gray}}

\put(15,33){\colorbox{gray}}
\put(15,37){\colorbox{gray}}
\put(15,42){\colorbox{gray}}
\put(19,33){\colorbox{gray}}
\put(19,37){\colorbox{gray}}
\put(19,42){\colorbox{gray}}
\put(24,33){\colorbox{gray}}
\put(24,37){\colorbox{gray}}
\put(24,42){\colorbox{gray}}

\put(30,33){\colorbox{gray}}
\put(30,37){\colorbox{gray}}
\put(30,42){\colorbox{gray}}
\put(34,33){\colorbox{gray}}
\put(34,37){\colorbox{gray}}
\put(34,42){\colorbox{gray}}
\put(39,33){\colorbox{gray}}
\put(39,37){\colorbox{gray}}
\put(39,42){\colorbox{gray}}

\put(45,33){\colorbox{gray}}
\put(45,37){\colorbox{gray}}
\put(45,42){\colorbox{gray}}
\put(49,33){\colorbox{gray}}
\put(49,37){\colorbox{gray}}
\put(49,42){\colorbox{gray}}
\put(54,33){\colorbox{gray}}
\put(54,37){\colorbox{gray}}
\put(54,42){\colorbox{gray}}

\put(60,33){\colorbox{gray}}
\put(60,37){\colorbox{gray}}
\put(60,42){\colorbox{gray}}
\put(64,33){\colorbox{gray}}
\put(64,37){\colorbox{gray}}
\put(64,42){\colorbox{gray}}
\put(69,33){\colorbox{gray}}
\put(69,37){\colorbox{gray}}
\put(69,42){\colorbox{gray}}

%
%
\put(30,18){\colorbox{gray}}
\put(30,22){\colorbox{gray}}
\put(30,27){\colorbox{gray}}
\put(34,18){\colorbox{gray}}
\put(34,22){\colorbox{gray}}
\put(34,27){\colorbox{gray}}
\put(39,18){\colorbox{gray}}
\put(39,22){\colorbox{gray}}
\put(39,27){\colorbox{gray}}

\put(45,18){\colorbox{gray}}
\put(45,22){\colorbox{gray}}
\put(45,27){\colorbox{gray}}
\put(49,18){\colorbox{gray}}
\put(49,22){\colorbox{gray}}
\put(49,27){\colorbox{gray}}
\put(54,18){\colorbox{gray}}
\put(54,22){\colorbox{gray}}
\put(54,27){\colorbox{gray}}

\put(60,18){\colorbox{gray}}
\put(60,22){\colorbox{gray}}
\put(60,27){\colorbox{gray}}
\put(64,18){\colorbox{gray}}
\put(64,22){\colorbox{gray}}
\put(64,27){\colorbox{gray}}
\put(69,18){\colorbox{gray}}
\put(69,22){\colorbox{gray}}
\put(69,27){\colorbox{gray}}

%
%
%
\put(45,3){\colorbox{gray}}
\put(45,7){\colorbox{gray}}
\put(45,12){\colorbox{gray}}
\put(49,3){\colorbox{gray}}
\put(49,7){\colorbox{gray}}
\put(49,12){\colorbox{gray}}
\put(54,3){\colorbox{gray}}
\put(54,7){\colorbox{gray}}
\put(54,12){\colorbox{gray}}

\put(60,3){\colorbox{gray}}
\put(60,7){\colorbox{gray}}
\put(60,12){\colorbox{gray}}
\put(64,3){\colorbox{gray}}
\put(64,7){\colorbox{gray}}
\put(64,12){\colorbox{gray}}
\put(69,3){\colorbox{gray}}
\put(69,7){\colorbox{gray}}
\put(69,12){\colorbox{gray}}

\put(0,0){\framebox(15,15)}
\put(15,0){\framebox(15,15)}
\put(30,0){\framebox(15,15)}
\put(45,0){\framebox(15,15)}
\put(60,0){\framebox(15,15)}
\put(0,15){\framebox(15,15)}
\put(15,15){\framebox(15,15)}
\put(30,15){\framebox(15,15)}
\put(45,15){\framebox(15,15)}
\put(60,15){\framebox(15,15)}
\put(0,30){\framebox(15,15)}
\put(15,30){\framebox(15,15)}
\put(30,30){\framebox(15,15)}
\put(45,30){\framebox(15,15)}
\put(60,30){\framebox(15,15)}
\put(0,45){\framebox(15,15)}
\put(15,45){\framebox(15,15)}
\put(30,45){\framebox(15,15)}
\put(45,45){\framebox(15,15)}
\put(60,45){\framebox(15,15)}
\put(0,60){\framebox(15,15)}
\put(15,60){\framebox(15,15)}
\put(30,60){\framebox(15,15)}
\put(45,60){\framebox(15,15)}
\put(60,60){\framebox(15,15)}

\put(100,63){\colorbox{gray}}
\put(100,67){\colorbox{gray}}
\put(100,72){\colorbox{gray}}
\put(104,63){\colorbox{gray}}
\put(104,67){\colorbox{gray}}
\put(104,72){\colorbox{gray}}
\put(109,63){\colorbox{gray}}
\put(109,67){\colorbox{gray}}
\put(109,72){\colorbox{gray}}

\put(115,63){\colorbox{gray}}
\put(115,67){\colorbox{gray}}
\put(115,72){\colorbox{gray}}
\put(119,63){\colorbox{gray}}
\put(119,67){\colorbox{gray}}
\put(119,72){\colorbox{gray}}
\put(124,63){\colorbox{gray}}
\put(124,67){\colorbox{gray}}
\put(124,72){\colorbox{gray}}

\put(130,63){\colorbox{gray}}
\put(130,67){\colorbox{gray}}
\put(130,72){\colorbox{gray}}
\put(134,63){\colorbox{gray}}
\put(134,67){\colorbox{gray}}
\put(134,72){\colorbox{gray}}
\put(139,63){\colorbox{gray}}
\put(139,67){\colorbox{gray}}
\put(139,72){\colorbox{gray}}

\put(145,63){\colorbox{gray}}
\put(145,67){\colorbox{gray}}
\put(145,72){\colorbox{gray}}
\put(149,63){\colorbox{gray}}
\put(149,67){\colorbox{gray}}
\put(149,72){\colorbox{gray}}
\put(154,63){\colorbox{gray}}
\put(154,67){\colorbox{gray}}
\put(154,72){\colorbox{gray}}

\put(160,63){\colorbox{gray}}
\put(160,67){\colorbox{gray}}
\put(160,72){\colorbox{gray}}
\put(164,63){\colorbox{gray}}
\put(164,67){\colorbox{gray}}
\put(164,72){\colorbox{gray}}
\put(169,63){\colorbox{gray}}
\put(169,67){\colorbox{gray}}
\put(169,72){\colorbox{gray}}

\put(100,48){\colorbox{gray}}
\put(100,52){\colorbox{gray}}
\put(100,57){\colorbox{gray}}
\put(104,48){\colorbox{gray}}
\put(104,52){\colorbox{gray}}
\put(104,57){\colorbox{gray}}
\put(109,48){\colorbox{gray}}
\put(109,52){\colorbox{gray}}
\put(109,57){\colorbox{gray}}

\put(115,48){\colorbox{gray}}
\put(115,52){\colorbox{gray}}
\put(115,57){\colorbox{gray}}
\put(119,48){\colorbox{gray}}
\put(119,52){\colorbox{gray}}
\put(119,57){\colorbox{gray}}
\put(124,48){\colorbox{gray}}
\put(124,52){\colorbox{gray}}
\put(124,57){\colorbox{gray}}

\put(130,48){\colorbox{gray}}
\put(130,52){\colorbox{gray}}
\put(130,57){\colorbox{gray}}
\put(134,48){\colorbox{gray}}
\put(134,52){\colorbox{gray}}
\put(134,57){\colorbox{gray}}
\put(139,48){\colorbox{gray}}
\put(139,52){\colorbox{gray}}
\put(139,57){\colorbox{gray}}

\put(145,48){\colorbox{gray}}
\put(145,52){\colorbox{gray}}
\put(145,57){\colorbox{gray}}
\put(149,48){\colorbox{gray}}
\put(149,52){\colorbox{gray}}
\put(149,57){\colorbox{gray}}
\put(154,48){\colorbox{gray}}
\put(154,52){\colorbox{gray}}
\put(154,57){\colorbox{gray}}

\put(160,48){\colorbox{gray}}
\put(160,52){\colorbox{gray}}
\put(160,57){\colorbox{gray}}
\put(164,48){\colorbox{gray}}
\put(164,52){\colorbox{gray}}
\put(164,57){\colorbox{gray}}
\put(169,48){\colorbox{gray}}
\put(169,52){\colorbox{gray}}
\put(169,57){\colorbox{gray}}

\put(100,33){\colorbox{gray}}
\put(100,37){\colorbox{gray}}
\put(100,42){\colorbox{gray}}
\put(104,33){\colorbox{gray}}
\put(104,37){\colorbox{gray}}
\put(104,42){\colorbox{gray}}
\put(109,33){\colorbox{gray}}
\put(109,37){\colorbox{gray}}
\put(109,42){\colorbox{gray}}

\put(115,33){\colorbox{gray}}
\put(115,37){\colorbox{gray}}
\put(115,42){\colorbox{gray}}
\put(119,33){\colorbox{gray}}
\put(119,37){\colorbox{gray}}
\put(119,42){\colorbox{gray}}
\put(124,33){\colorbox{gray}}
\put(124,37){\colorbox{gray}}
\put(124,42){\colorbox{gray}}

\put(130,33){\colorbox{gray}}
\put(130,37){\colorbox{gray}}
\put(130,42){\colorbox{gray}}
\put(134,33){\colorbox{gray}}
\put(134,37){\colorbox{gray}}
\put(134,42){\colorbox{gray}}
\put(139,33){\colorbox{gray}}
\put(139,37){\colorbox{gray}}
\put(139,42){\colorbox{gray}}

\put(145,33){\colorbox{gray}}
\put(145,37){\colorbox{gray}}
\put(145,42){\colorbox{gray}}
\put(149,33){\colorbox{gray}}
\put(149,37){\colorbox{gray}}
\put(149,42){\colorbox{gray}}
\put(154,33){\colorbox{gray}}
\put(154,37){\colorbox{gray}}
\put(154,42){\colorbox{gray}}

\put(160,33){\colorbox{gray}}
\put(160,37){\colorbox{gray}}
\put(160,42){\colorbox{gray}}
\put(164,33){\colorbox{gray}}
\put(164,37){\colorbox{gray}}
\put(164,42){\colorbox{gray}}
\put(169,33){\colorbox{gray}}
\put(169,37){\colorbox{gray}}
\put(169,42){\colorbox{gray}}

%
\put(115,18){\colorbox{Gray}}
\put(115,22){\colorbox{Gray}}
\put(115,27){\colorbox{Gray}}
\put(119,18){\colorbox{Gray}}
\put(119,22){\colorbox{Gray}}
\put(119,27){\colorbox{Gray}}
\put(124,18){\colorbox{Gray}}
\put(124,22){\colorbox{Gray}}
\put(124,27){\colorbox{Gray}}

\put(130,18){\colorbox{gray}}
\put(130,22){\colorbox{gray}}
\put(130,27){\colorbox{gray}}
\put(134,18){\colorbox{gray}}
\put(134,22){\colorbox{gray}}
\put(134,27){\colorbox{gray}}
\put(139,18){\colorbox{gray}}
\put(139,22){\colorbox{gray}}
\put(139,27){\colorbox{gray}}

\put(145,18){\colorbox{gray}}
\put(145,22){\colorbox{gray}}
\put(145,27){\colorbox{gray}}
\put(149,18){\colorbox{gray}}
\put(149,22){\colorbox{gray}}
\put(149,27){\colorbox{gray}}
\put(154,18){\colorbox{gray}}
\put(154,22){\colorbox{gray}}
\put(154,27){\colorbox{gray}}

\put(160,18){\colorbox{gray}}
\put(160,22){\colorbox{gray}}
\put(160,27){\colorbox{gray}}
\put(164,18){\colorbox{gray}}
\put(164,22){\colorbox{gray}}
\put(164,27){\colorbox{gray}}
\put(169,18){\colorbox{gray}}
\put(169,22){\colorbox{gray}}
\put(169,27){\colorbox{gray}}

%
%
%
\put(145,3){\colorbox{gray}}
\put(145,7){\colorbox{gray}}
\put(145,12){\colorbox{gray}}
\put(149,3){\colorbox{gray}}
\put(149,7){\colorbox{gray}}
\put(149,12){\colorbox{gray}}
\put(154,3){\colorbox{gray}}
\put(154,7){\colorbox{gray}}
\put(154,12){\colorbox{gray}}

\put(160,3){\colorbox{gray}}
\put(160,7){\colorbox{gray}}
\put(160,12){\colorbox{gray}}
\put(164,3){\colorbox{gray}}
\put(164,7){\colorbox{gray}}
\put(164,12){\colorbox{gray}}
\put(169,3){\colorbox{gray}}
\put(169,7){\colorbox{gray}}
\put(169,12){\colorbox{gray}}

\put(100,0){\framebox(15,15)}
\put(115,0){\framebox(15,15)}
\put(130,0){\framebox(15,15)}
\put(145,0){\framebox(15,15)}
\put(160,0){\framebox(15,15)}
\put(100,15){\framebox(15,15)}
\put(115,15){\framebox(15,15)}
\put(130,15){\framebox(15,15)}
\put(145,15){\framebox(15,15)}
\put(160,15){\framebox(15,15)}
\put(100,30){\framebox(15,15)}
\put(115,30){\framebox(15,15)}
\put(130,30){\framebox(15,15)}
\put(145,30){\framebox(15,15)}
\put(160,30){\framebox(15,15)}
\put(100,45){\framebox(15,15)}
\put(115,45){\framebox(15,15)}
\put(130,45){\framebox(15,15)}
\put(145,45){\framebox(15,15)}
\put(160,45){\framebox(15,15)}
\put(100,60){\framebox(15,15)}
\put(115,60){\framebox(15,15)}
\put(130,60){\framebox(15,15)}
\put(145,60){\framebox(15,15)}
\put(160,60){\framebox(15,15)}

\put(180,20){$\leftarrow (h(r)+1)$-th row}
\put(120,80){$\downarrow$}
\put(115,90){$r$-th column}

\put(35,-15){$h$}
\put(135,-15){$h'$}
\end{picture}
\end{center}
\vspace{5pt}
\caption{The pictures of $h$ and $h'$.}
\end{figure}
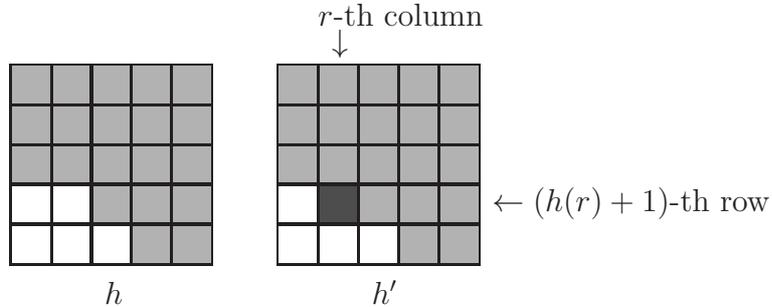

To compare the Poincar\'{e} polynomials $\Ph$ and $\Phh$, 
let us consider the following polynomials:
\begin{align} 
&P_{+}^{h,h'}(\q)=\displaystyle\sum_{i\geq 0} |\{w\in \Sn \mid d_h(w)=i, \ w(r)<w(h'(r)) \}| \q^i, \label{eq:defP+} \\
&P_{-}^{h,h'}(\q)=\displaystyle\sum_{i\geq 0} |\{w\in \Sn \mid d_h(w)=i, \ w(r)>w(h'(r)) \}| \q^i. \label{eq:defP-} 
\end{align}
As we will see below, these polynomials play an important role for understanding the difference of $\Ph$ and $\Phh$.
From Theorem~\ref{theorem:Xh}~(3), it is easy to see that
\begin{align} \label{eq:Ph}
\Ph=P_{+}^{h,h'}(\q)+P_{-}^{h,h'}(\q), \quad \Phh=P_{+}^{h,h'}(\q)+\q P_{-}^{h,h'}(\q)
\end{align}
which imply
\begin{align}\label{eq:Ph-Phh}
\Ph-\Phh=(1-\q) P_{-}^{h,h'}(\q). 
\end{align}
On the other hand, from the Poincar\'{e} duality for $\Xh$ and $\Xhh$ (Theorem~\ref{theorem:Xh}~(1)), we have the following symmetry
$$
\Ph=\q^{\dimX}P(\Xh,\q^{-1}), \quad \Phh=\q^{\dimX+1}P(\Xhh,\q^{-1})
$$
where $\dimX=\dim _{\C}\Xh$. 
Multiplying \eqref{eq:Ph-Phh} for $q$ being replaced with $q^{-1}$ by $q^{\dimX+1}$, we obtain
\begin{equation}\label{eq:qPh-Phh}
\q \Ph -\Phh=\q^{\dimX}(\q-1)P_{-}^{h,h'}(\q^{-1}).
\end{equation}
Now, from \eqref{eq:Ph-Phh} and \eqref{eq:qPh-Phh}, the Poincar\'{e} polynomial $\Ph$ can be written as
$$
\Ph=P_{-}^{h,h'}(\q)+\q^{\dimX}P_{-}^{h,h'}(\q^{-1}).
$$
Comparing this with \eqref{eq:Ph}, we obtain
\begin{align} \label{eq:P+-}
P_{+}^{h,h'}(\q)=\q^{\dimX} P_{-}^{h,h'}(\q^{-1}), \ {\rm equivalently}  \ P_{-}^{h,h'}(\q)=\q^{\dimX}P_{+}^{h,h'}(\q^{-1}).
\end{align}
This means that the Poincare polynomial $\Ph$ is determined from $P_{+}^{h,h'}(\q)$. The next lemma tells us how this polynomial $P_{+}^{h,h'}(\q)$ behaves when we further add a new box to the $r$-th column.

\begin{lemma}\label{lemma:PoincarePolynomial}
Let $h$ be a Hessenberg function satisfying the condition $h(r)+1<h(r+1)$ for some $1 \leq r<n$ so that
\begin{align*}
h'&:=(h(1),\ldots,h(r-1),h(r)+1,h(r+1),\ldots,h(n)), \\
h''&:=(h(1),\ldots,h(r-1),h(r)+2,h(r+1),\ldots,h(n))
\end{align*} 
are Hessenberg functions. If $h(h'(r))=h(h''(r))$, then $P_{+}^{h,h'}(\q)=P_{+}^{h',h''}(\q)$.
\end{lemma}

\begin{proof}
For simplicity, we put 
$$
A:=\{w\in \Sn \mid w(r)<w(h'(r)) \}, \ B:=\{w\in \Sn \mid w(r)<w(h''(r)) \}.
$$
Then, from the definition \eqref{eq:defP+}, our claim is equivalent to the equality
$$
|\{w\in A \mid d_h(w)=i \}|=|\{w\in B \mid d_{h'}(w)=i \}| \ \ \ {\rm for \ each} \ i.
$$
Note first that the transposition $\tau:=(h'(r) \ h''(r))$ interchanging $h'(r)$ and $h''(r)$ makes sense since we have $h'(r)<h''(r)$ by our assumption.
So the multiplication map of $\Sn$ by the transposition $\tau$ from the right gives a bijection from $A$ to $B$ preserving the intersection $A \cap B$. Thus, it is enough to prove the following two claims:
\begin{itemize}
\item[Case 1] : if $w\in A\cap B$, then $d_h(w)=d_{h'}(w)$, 
\item[Case 2] : if $w\in A\setminus A\cap B$, then $w':=w\tau\in B\setminus A\cap B$ satisfies $d_{h}(w)=d_{h'}(w')$.
\end{itemize}

\smallskip

\noindent
Proof of Case 1. If $w\in A\cap B$, then $w(r)<w(h'(r))$, so we have $d_h(w)=d_{h'}(w)$. \\

\noindent
Proof of Case 2. Let $w\in A \setminus A\cap B$ and $w'=w\tau\in B \setminus A\cap B$. Noting that $h''(r)(=h(r)+2)\leq h(r+1)$ from the assumption, 
we can write $w$ and $w'$ in one-line notation as follows:
\begin{align*}
w&=w(1)\cdots w(r) \cdots w(h(r)) \ w(h'(r)) \ w(h''(r)) \cdots w(h(r+1)) \cdots w(n), \\
w'&=w(1)\cdots w(r) \cdots w(h(r)) \ w(h''(r)) \ w(h'(r)) \cdots w(h(r+1)) \cdots w(n).
\end{align*}
The claim we need to show is $|D_{h}(w)|=|D_{h'}(w')|$ where $D_{h}(w)$ is defined in Theorem \ref{theorem:Xh} (2). To prove this, we compare these sets $D_h(w)$ and $D_{h'}(w')$. 
Note that $h$ and $h'$ differ only at $r$-th values and that $w'$ is obtained from $w$ by interchanging $w(h'(r))$ and $w(h''(r))$ in one-line notation as we saw above.
So, with the assumption $h(h'(r))=h(h''(r))$, it follows that $(i,j)\in D_{h}(w)$ is equivalent to $(\tau(i),\tau(j))\in D_{h'}(w')$ unless $(i,j)$ is one of the following:
\begin{align}\label{eq:three pairs}
(r,h'(r)), \ \ (r,h''(r)), \ \ (h'(r),h''(r)).
\end{align}
So we focus on these pairs in what follows.
Since $w\in A \setminus A\cap B$ and $w'\in B \setminus A\cap B$, we have 
$$
w(h''(r))<w(r)<w(h'(r)).
$$
This means that the third pair $(h'(r),h''(r))$ in \eqref{eq:three pairs} is an element of $D_h(w)$ since we have $h''(r)\leq h(h''(r))=h(h'(r))$ by the assumption, but it is not an element of $D_{h'}(w')$ since $w'$ is obtained from $w$ by replacing $w(h'(r))$ and $w(h''(r))$ in the one-line notation of $w$. Namely, we have $$(h'(r),h''(r)) \in D_h(w)\setminus D_{h'}(w').$$
The first pair $(r,h'(r))$ in \eqref{eq:three pairs} is not an element of $D_h(w)$ since $h'(r)\not\leq h(r)$, but it is an element of $D_{h'}(w')$ since $w'(r)=w(r)\not< w(h''(r))=w'(h'(r))$. So we have
$$(r,h'(r))\in D_{h'}(w')\setminus D_h(w).$$
Lastly, the second pair $(r,h''(r))$ in \eqref{eq:three pairs} is not an element of neither $D_h(w)$ nor $D_{h'}(w')$ since we have $h''(r)\not\leq h(r)$ and $w'(r)=w(r)\not> w(h'(r))=w'(h''(r))$, i.e.
\begin{align*}
(r,h''(r)) \notin D_h(w)\cup D_{h'}(w').
\end{align*}
Therefore, we conclude that $|D_{h}(w)|=|D_{h'}(w')|$, as desired.
\end{proof}

Now we can express an arithmetic formula for the Poincar\'{e} polynomial $\Ph$ (cf. Theorem~\ref{theorem:Xh}~(3)) whose Hessenberg function is of the form $h=(h(1),n,\ldots,n)$.

\begin{lemma}\label{lemma:PoincarePolynomailXh}
If $h=(h(1),n,\ldots,n)$, then the Poincar\'{e} polynomial of $\Xh$ is given by
\begin{align} \label{eq:Ph1}
\Ph=\frac{1-\q^{h(1)}}{1-\q}\displaystyle\prod_{j=1}^{n-1} \frac{1-\q^j}{1-\q}+ (n-1)\q^{h(1)-1}\frac{1-\q^{n-h(1)}}{1-\q}\displaystyle\prod_{j=1}^{n-2} \frac{1-\q^j}{1-\q}.
\end{align}
\end{lemma}

\begin{proof}
Since it is well-known that 
\begin{align*}
P(\Flags(\C^n),q)=\displaystyle\prod_{j=1}^{n} \frac{1-\q^j}{1-\q}
\end{align*}
which gives the claim \eqref{eq:Ph1} for the case $h(1)=n$, we can assume that $h(1)<n$.
We first prove that 
\begin{align} \label{eq:P+}
P_{+}^{h,h'}(\q)=\big(\displaystyle\sum_{k=1}^{n-1} k \q^{k-1} \big) \displaystyle\prod_{j=1}^{n-2} \frac{1-\q^j}{1-\q}
\end{align}
for $h=(h(1),n,\ldots,n)$ and $h'=(h(1)+1,n,\ldots,n)$. 
From Lemma~\ref{lemma:PoincarePolynomial}, it is enough to prove \eqref{eq:P+} for $h(1)=n-1$ (i.e. $h'(1)=n$).
We fix integers $a$ and $b$ with $1 \leq a<b \leq n$. 
For $w\in\Sn$ satisfying $w(1)=a$ and $w(n)=b$, we have
\begin{align*}
 d_h(w) 
 =& \ |\{ (j,i) \mid 2\leq j<i\leq n-1, \ w(j)>w(i) \}| \\ 
 &\qquad+ |\{ (1,i) \mid 1<i<n, \ a>w(i) \}| \\
 &\qquad+ |\{ (j,n) \mid 1<j<n, \ w(j)>b \}| \\
 =& \ |\{ (j,i) \mid 2\leq j<i\leq n-1, \ w(j)>w(i) \}| +(a-1) + (n-b).
\end{align*}
So it follows that
\begin{align*}
\displaystyle\sum_{ w \in \Sn \atop w(1)=a, w(n)=b} \q^{d_h(w)}&=\q^{(a-1)+(n-b)}\displaystyle \sum_{ v \in \S_{n-2}} \q^{\ell(v)} 
\\
&=\q^{(a-1)+(n-b)}P(\Flags(\C^{n-2}),\q) \notag \\ 
&=\q^{(a-1)+(n-b)}\displaystyle\prod_{j=1}^{n-2} \frac{1-\q^j}{1-\q}. \notag
\end{align*}
where $\ell(v)$ is the number of inversions of $v\in\S_{n-2}$.
Thus, from \eqref{eq:defP+} we obtain
\begin{align*}
P_{+}^{h,h'}(\q)&=\displaystyle\sum_{1 \leq a<b \leq n} \q^{(a-1)+(n-b)}\displaystyle\prod_{j=1}^{n-2} \frac{1-\q^j}{1-\q} \\
                  &=\big(\displaystyle\sum_{1 \leq a<b \leq n} \q^{a-b}\big)\q^{n-1}\displaystyle\prod_{j=1}^{n-2} \frac{1-\q^j}{1-\q} \\
                  &=\big(\displaystyle\sum_{k=1}^{n-1} k \q^{k-1} \big) \displaystyle\prod_{j=1}^{n-2} \frac{1-\q^j}{1-\q}
\end{align*}
which is exactly \eqref{eq:P+} as we claimed.
Hence, from \eqref{eq:P+-} we obtain 
\begin{align*} 
P_{-}^{h,h'}(\q)=\q^{h(1)-1}\big(\displaystyle\sum_{k=1}^{n-1} (n-k) \q^{k-1} \big) \displaystyle\prod_{j=1}^{n-2} \frac{1-\q^j}{1-\q}.
\end{align*}
Now, the claim \eqref{eq:Ph1} follows from \eqref{eq:Ph}.
\end{proof}

\bigskip

\section{Cohomology rings of $\Xh$ for $h=(h(1),n,\ldots,n)$}

In this section, we give an explicit presentation of the cohomology rings $H^*(\Xh)$ of regular semisimple Hessenberg varieties $\Xh$ for $h=(h(1),n,\ldots,n)$ in terms of ring generators and their relations.
Through this presentation, we will determine the $\Sn$-representation on $H^*(\Xh;\C)$ in Section \ref{subsec:Sn-rep}.
We will also give an explicit presentation of the $\Sn$-invariant subring $H^*(\Xh)^{\Sn}$ with $\Z$-coefficients in Section \ref{subsec:inv-subring}, and compare with the fact that $H^*(\Xh;\Q)^{\Sn}$ is isomorphic to the cohomology ring of regular nilpotent Hessenberg variety with the same Hessenberg function in $\Q$-coefficients.

\subsection{The ring structure of $H^*(\Xh)$}

As explained in Section \ref{sec:background}, the restriction map $H^*(\Flags(\C^n))\rightarrow H^*(\Xh)$ is not surjective in general, and it is not obvious how to find ring generators of $H^*(\Xh)$. For this purpose, we first study the $\Tn$-equivariant cohomology ring $H^*_{\Tn}(\Xh)$ through its graphical presentation given in Proposition \ref{proposition:GKM}, and we study certain classes which do not come from $H^*_{\Tn}(\Flags(\C^n))$ in general.

For any Hessenberg function $h$, we first define classes $\x_k$ and $\y_{k}$ in the equivariant cohomology $H^*_{\Tn}(\Xh)$ for $1\leq k \leq n$ as follows. For $w\in \Sn$, let\footnote[2]{The class $\y_{k}$ defined here is the $\y_{j,k} \in H^{2(h(j)-j)}_{\Tn}(\Xh)$ for $j=1$ introduced by \cite{AHHM}.}
\begin{align}
\x_k(w)&:=t_{w(k)}, \label{eq:defx} \\
\y_{k}(w)&:=\begin{cases}\prod_{\ell=2}^{h(1)} (t_k-t_{w(\ell)}) \ \ &\mbox{if} \ w(1)=k \label{eq:defy} \\ 
0  &\mbox{if} \ w(1)\neq k \end{cases}
\end{align}
where we take the convention $\prod_{\ell=2}^{h(1)} (t_k-t_{w(\ell)})=1$ when $h(1)=1$.
Then $x_k=(x_k(w))_{w\in \Sn}$ and $y_{k}=(y_{k}(w))_{w\in \Sn}$ satisfy the condition in \eqref{eq:GKM} (see \cite[Lemma~10.2]{AHHM}) so that we have $x_k \in H^2_{\Tn}(\Xh)$ and $y_{k}\in H^{2(h(1)-1)}_{\Tn}(\Xh)$. 
Note that the class $y_k$ is supported on the set of permutations $w\in\Sn$ with $w(1)=k$, while $x_k$ is non-zero on any $w$-component. Note that the $x_k$ is the image of the $\Tn$-equivariant first Chern class of the $k$-th tautological line bundle over $\Flags(\C^n)$ under the restriction map $H^*_{\Tn}(\Flags(\C^n)) \to H^*_{\Tn}(\Xh)$. 
It is well-known that 
\begin{equation} \label{eq:CohomologyFlag}
H^*_{\Tn}(\Flags(\C^n)) \cong \Z[X_1,\ldots,X_n,t_1,\ldots,t_n]/(e_i(X)-e_i(t) \mid 1\leq i \leq n)
\end{equation}
by sending $X_k$ to $\x_k$ and $t_k$ to $t_k$ for all $k=1,\ldots,n$.
Here, $e_i(X)$ (resp. $e_i(t)$) is the $i$-th elementary symmetric polynomial in the variables $X_1,\ldots,X_n$ (resp. $t_1,\ldots,t_n$).
In particular, we have
\begin{equation} \label{eq:CohomologyFlag2}
H^*(\Flags(\C^n)) \cong \Z[X_1,\ldots,X_n]/(e_i(X) \mid 1\leq i \leq n).
\end{equation}
We list some of the basic relations between the classes $x_k$ and $\y_{k'}$ in the following lemma.

\begin{lemma} \label{lemma:yk}
The following hold:
\begin{itemize}
\item[$(1)$] $\y_{k}\y_{k'}=0$ \ \ $(1\leq k\neq k'\leq n)$, 
\item[$(2)$] $\x_1\y_{k}=t_k\y_{k}$ \ \ $(1 \leq k \leq n)$,
\item[$(3)$] $\y_{k}\prod_{\ell=h(1)+1}^{n} (t_k-\x_{\ell})=\prod_{\ell=2}^{n} (t_k-\x_{\ell})$ \ \ $(1 \leq k \leq n)$,   
\item[$(4)$] $\sum_{k=1}^{n}\y_{k}=\prod_{\ell=2}^{h(1)} (\x_1-\x_{\ell})$
\end{itemize}
where we take the convention $\prod_{\ell=n+1}^{n} (-x_{\ell})=1$ in $(3)$ and $\prod_{\ell=2}^{1} (x_1-x_{\ell})=1$ in $(4)$. 
\end{lemma}

\begin{proof}
It is easy to check that the $w$-components of both sides are the same for all $w\in \Sn$. 
Since the restriction map $H^*_{\Tn}(\Xh) \to H^*_{\Tn}(\Xh^{\Tn})=\bigoplus _{w \in \Sn} \Z[t_1,\dots,t_n]$ is injective, we are done. 
\end{proof}

Multiplying both sides of (4) in Lemma~\ref{lemma:yk} by $\y_k$, we also have
\begin{equation} \label{eq:yksquare}
\y_k^2=\y_k \displaystyle\prod_{\ell=2}^{h(1)}(t_k-\x_{\ell})
\end{equation}
by (1) and (2) in Lemma~\ref{lemma:yk}.

\begin{proposition}\label{proposition:generator}
If $h=(h(1),n,\ldots,n)$, then the classes $\x_k, \y_k, t_k$ $(1\leq k \leq n)$ generate $H^{*}_{\Tn}(\Xh)$ as a  $\Z$-algebra. 
\end{proposition}

\begin{proof}
We prove this proposition by induction on $n$ along the idea in \cite{fu-is-ma}. 
We decompose the symmetric group $\Sn$ into the following subsets for $i=1,\ldots,n$:
$$
\Sn ^i:=\{w\in \Sn \mid w(i)=n\}.
$$
Namely, we have $\Sn=\bigcup_{i=1}^n\Sn^i$ which is a disjoint union.
Let $\Gamma (h)$ be the GKM graph of $\Xh$ described in Section \ref{sec:background}. 
We denote by $\Gamma ^{i}(h)$ the full subgraph of $\Gamma (h)$ whose vertex set is $\Sn^i$. 
We think of $h$ as a configuration of the boxes as explained in  Section \ref{subsec:background}, and 
let $h^{i}$ be a Hessenberg function
obtained by removing all the boxes in the $i$-th row and all the boxes in the $i$-th column (See Figure \ref{pic:h^i}). Then $\Gamma ^{i}(h)$ is naturally identified with $\Gamma (h^{i})$ so that we can apply the inductive assumption to $\Gamma ^{i}(h)$. 

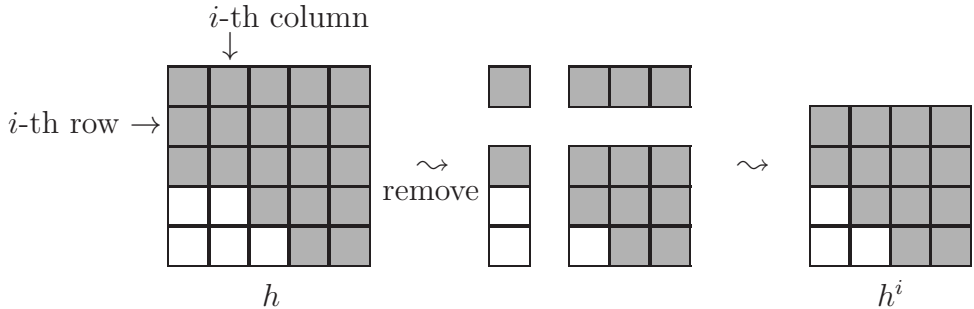
\begin{figure}[h]
\begin{center}
\begin{picture}(240,100)
\put(0,63){\colorbox{gray}}
\put(0,67){\colorbox{gray}}
\put(0,72){\colorbox{gray}}
\put(4,63){\colorbox{gray}}
\put(4,67){\colorbox{gray}}
\put(4,72){\colorbox{gray}}
\put(9,63){\colorbox{gray}}
\put(9,67){\colorbox{gray}}
\put(9,72){\colorbox{gray}}

\put(15,63){\colorbox{gray}}
\put(15,67){\colorbox{gray}}
\put(15,72){\colorbox{gray}}
\put(19,63){\colorbox{gray}}
\put(19,67){\colorbox{gray}}
\put(19,72){\colorbox{gray}}
\put(24,63){\colorbox{gray}}
\put(24,67){\colorbox{gray}}
\put(24,72){\colorbox{gray}}

\put(30,63){\colorbox{gray}}
\put(30,67){\colorbox{gray}}
\put(30,72){\colorbox{gray}}
\put(34,63){\colorbox{gray}}
\put(34,67){\colorbox{gray}}
\put(34,72){\colorbox{gray}}
\put(39,63){\colorbox{gray}}
\put(39,67){\colorbox{gray}}
\put(39,72){\colorbox{gray}}

\put(45,63){\colorbox{gray}}
\put(45,67){\colorbox{gray}}
\put(45,72){\colorbox{gray}}
\put(49,63){\colorbox{gray}}
\put(49,67){\colorbox{gray}}
\put(49,72){\colorbox{gray}}
\put(54,63){\colorbox{gray}}
\put(54,67){\colorbox{gray}}
\put(54,72){\colorbox{gray}}

\put(60,63){\colorbox{gray}}
\put(60,67){\colorbox{gray}}
\put(60,72){\colorbox{gray}}
\put(64,63){\colorbox{gray}}
\put(64,67){\colorbox{gray}}
\put(64,72){\colorbox{gray}}
\put(69,63){\colorbox{gray}}
\put(69,67){\colorbox{gray}}
\put(69,72){\colorbox{gray}}

\put(0,48){\colorbox{gray}}
\put(0,52){\colorbox{gray}}
\put(0,57){\colorbox{gray}}
\put(4,48){\colorbox{gray}}
\put(4,52){\colorbox{gray}}
\put(4,57){\colorbox{gray}}
\put(9,48){\colorbox{gray}}
\put(9,52){\colorbox{gray}}
\put(9,57){\colorbox{gray}}

\put(15,48){\colorbox{gray}}
\put(15,52){\colorbox{gray}}
\put(15,57){\colorbox{gray}}
\put(19,48){\colorbox{gray}}
\put(19,52){\colorbox{gray}}
\put(19,57){\colorbox{gray}}
\put(24,48){\colorbox{gray}}
\put(24,52){\colorbox{gray}}
\put(24,57){\colorbox{gray}}

\put(30,48){\colorbox{gray}}
\put(30,52){\colorbox{gray}}
\put(30,57){\colorbox{gray}}
\put(34,48){\colorbox{gray}}
\put(34,52){\colorbox{gray}}
\put(34,57){\colorbox{gray}}
\put(39,48){\colorbox{gray}}
\put(39,52){\colorbox{gray}}
\put(39,57){\colorbox{gray}}

\put(45,48){\colorbox{gray}}
\put(45,52){\colorbox{gray}}
\put(45,57){\colorbox{gray}}
\put(49,48){\colorbox{gray}}
\put(49,52){\colorbox{gray}}
\put(49,57){\colorbox{gray}}
\put(54,48){\colorbox{gray}}
\put(54,52){\colorbox{gray}}
\put(54,57){\colorbox{gray}}

\put(60,48){\colorbox{gray}}
\put(60,52){\colorbox{gray}}
\put(60,57){\colorbox{gray}}
\put(64,48){\colorbox{gray}}
\put(64,52){\colorbox{gray}}
\put(64,57){\colorbox{gray}}
\put(69,48){\colorbox{gray}}
\put(69,52){\colorbox{gray}}
\put(69,57){\colorbox{gray}}

\put(0,33){\colorbox{gray}}
\put(0,37){\colorbox{gray}}
\put(0,42){\colorbox{gray}}
\put(4,33){\colorbox{gray}}
\put(4,37){\colorbox{gray}}
\put(4,42){\colorbox{gray}}
\put(9,33){\colorbox{gray}}
\put(9,37){\colorbox{gray}}
\put(9,42){\colorbox{gray}}

\put(15,33){\colorbox{gray}}
\put(15,37){\colorbox{gray}}
\put(15,42){\colorbox{gray}}
\put(19,33){\colorbox{gray}}
\put(19,37){\colorbox{gray}}
\put(19,42){\colorbox{gray}}
\put(24,33){\colorbox{gray}}
\put(24,37){\colorbox{gray}}
\put(24,42){\colorbox{gray}}

\put(30,33){\colorbox{gray}}
\put(30,37){\colorbox{gray}}
\put(30,42){\colorbox{gray}}
\put(34,33){\colorbox{gray}}
\put(34,37){\colorbox{gray}}
\put(34,42){\colorbox{gray}}
\put(39,33){\colorbox{gray}}
\put(39,37){\colorbox{gray}}
\put(39,42){\colorbox{gray}}

\put(45,33){\colorbox{gray}}
\put(45,37){\colorbox{gray}}
\put(45,42){\colorbox{gray}}
\put(49,33){\colorbox{gray}}
\put(49,37){\colorbox{gray}}
\put(49,42){\colorbox{gray}}
\put(54,33){\colorbox{gray}}
\put(54,37){\colorbox{gray}}
\put(54,42){\colorbox{gray}}

\put(60,33){\colorbox{gray}}
\put(60,37){\colorbox{gray}}
\put(60,42){\colorbox{gray}}
\put(64,33){\colorbox{gray}}
\put(64,37){\colorbox{gray}}
\put(64,42){\colorbox{gray}}
\put(69,33){\colorbox{gray}}
\put(69,37){\colorbox{gray}}
\put(69,42){\colorbox{gray}}

%
%
\put(30,18){\colorbox{gray}}
\put(30,22){\colorbox{gray}}
\put(30,27){\colorbox{gray}}
\put(34,18){\colorbox{gray}}
\put(34,22){\colorbox{gray}}
\put(34,27){\colorbox{gray}}
\put(39,18){\colorbox{gray}}
\put(39,22){\colorbox{gray}}
\put(39,27){\colorbox{gray}}

\put(45,18){\colorbox{gray}}
\put(45,22){\colorbox{gray}}
\put(45,27){\colorbox{gray}}
\put(49,18){\colorbox{gray}}
\put(49,22){\colorbox{gray}}
\put(49,27){\colorbox{gray}}
\put(54,18){\colorbox{gray}}
\put(54,22){\colorbox{gray}}
\put(54,27){\colorbox{gray}}

\put(60,18){\colorbox{gray}}
\put(60,22){\colorbox{gray}}
\put(60,27){\colorbox{gray}}
\put(64,18){\colorbox{gray}}
\put(64,22){\colorbox{gray}}
\put(64,27){\colorbox{gray}}
\put(69,18){\colorbox{gray}}
\put(69,22){\colorbox{gray}}
\put(69,27){\colorbox{gray}}

%
%
%
\put(45,3){\colorbox{gray}}
\put(45,7){\colorbox{gray}}
\put(45,12){\colorbox{gray}}
\put(49,3){\colorbox{gray}}
\put(49,7){\colorbox{gray}}
\put(49,12){\colorbox{gray}}
\put(54,3){\colorbox{gray}}
\put(54,7){\colorbox{gray}}
\put(54,12){\colorbox{gray}}

\put(60,3){\colorbox{gray}}
\put(60,7){\colorbox{gray}}
\put(60,12){\colorbox{gray}}
\put(64,3){\colorbox{gray}}
\put(64,7){\colorbox{gray}}
\put(64,12){\colorbox{gray}}
\put(69,3){\colorbox{gray}}
\put(69,7){\colorbox{gray}}
\put(69,12){\colorbox{gray}}

\put(0,0){\framebox(15,15)}
\put(15,0){\framebox(15,15)}
\put(30,0){\framebox(15,15)}
\put(45,0){\framebox(15,15)}
\put(60,0){\framebox(15,15)}
\put(0,15){\framebox(15,15)}
\put(15,15){\framebox(15,15)}
\put(30,15){\framebox(15,15)}
\put(45,15){\framebox(15,15)}
\put(60,15){\framebox(15,15)}
\put(0,30){\framebox(15,15)}
\put(15,30){\framebox(15,15)}
\put(30,30){\framebox(15,15)}
\put(45,30){\framebox(15,15)}
\put(60,30){\framebox(15,15)}
\put(0,45){\framebox(15,15)}
\put(15,45){\framebox(15,15)}
\put(30,45){\framebox(15,15)}
\put(45,45){\framebox(15,15)}
\put(60,45){\framebox(15,15)}
\put(0,60){\framebox(15,15)}
\put(15,60){\framebox(15,15)}
\put(30,60){\framebox(15,15)}
\put(45,60){\framebox(15,15)}
\put(60,60){\framebox(15,15)}

\put(120,63){\colorbox{gray}}
\put(120,67){\colorbox{gray}}
\put(120,72){\colorbox{gray}}
\put(124,63){\colorbox{gray}}
\put(124,67){\colorbox{gray}}
\put(124,72){\colorbox{gray}}
\put(129,63){\colorbox{gray}}
\put(129,67){\colorbox{gray}}
\put(129,72){\colorbox{gray}}


\put(150,63){\colorbox{gray}}
\put(150,67){\colorbox{gray}}
\put(150,72){\colorbox{gray}}
\put(154,63){\colorbox{gray}}
\put(154,67){\colorbox{gray}}
\put(154,72){\colorbox{gray}}
\put(159,63){\colorbox{gray}}
\put(159,67){\colorbox{gray}}
\put(159,72){\colorbox{gray}}

\put(165,63){\colorbox{gray}}
\put(165,67){\colorbox{gray}}
\put(165,72){\colorbox{gray}}
\put(169,63){\colorbox{gray}}
\put(169,67){\colorbox{gray}}
\put(169,72){\colorbox{gray}}
\put(174,63){\colorbox{gray}}
\put(174,67){\colorbox{gray}}
\put(174,72){\colorbox{gray}}

\put(180,63){\colorbox{gray}}
\put(180,67){\colorbox{gray}}
\put(180,72){\colorbox{gray}}
\put(184,63){\colorbox{gray}}
\put(184,67){\colorbox{gray}}
\put(184,72){\colorbox{gray}}
\put(189,63){\colorbox{gray}}
\put(189,67){\colorbox{gray}}
\put(189,72){\colorbox{gray}}

%
%
%
%

\put(120,33){\colorbox{gray}}
\put(120,37){\colorbox{gray}}
\put(120,42){\colorbox{gray}}
\put(124,33){\colorbox{gray}}
\put(124,37){\colorbox{gray}}
\put(124,42){\colorbox{gray}}
\put(129,33){\colorbox{gray}}
\put(129,37){\colorbox{gray}}
\put(129,42){\colorbox{gray}}


\put(150,33){\colorbox{gray}}
\put(150,37){\colorbox{gray}}
\put(150,42){\colorbox{gray}}
\put(154,33){\colorbox{gray}}
\put(154,37){\colorbox{gray}}
\put(154,42){\colorbox{gray}}
\put(159,33){\colorbox{gray}}
\put(159,37){\colorbox{gray}}
\put(159,42){\colorbox{gray}}

\put(165,33){\colorbox{gray}}
\put(165,37){\colorbox{gray}}
\put(165,42){\colorbox{gray}}
\put(169,33){\colorbox{gray}}
\put(169,37){\colorbox{gray}}
\put(169,42){\colorbox{gray}}
\put(174,33){\colorbox{gray}}
\put(174,37){\colorbox{gray}}
\put(174,42){\colorbox{gray}}

\put(180,33){\colorbox{gray}}
\put(180,37){\colorbox{gray}}
\put(180,42){\colorbox{gray}}
\put(184,33){\colorbox{gray}}
\put(184,37){\colorbox{gray}}
\put(184,42){\colorbox{gray}}
\put(189,33){\colorbox{gray}}
\put(189,37){\colorbox{gray}}
\put(189,42){\colorbox{gray}}

%

\put(150,18){\colorbox{gray}}
\put(150,22){\colorbox{gray}}
\put(150,27){\colorbox{gray}}
\put(154,18){\colorbox{gray}}
\put(154,22){\colorbox{gray}}
\put(154,27){\colorbox{gray}}
\put(159,18){\colorbox{gray}}
\put(159,22){\colorbox{gray}}
\put(159,27){\colorbox{gray}}

\put(165,18){\colorbox{gray}}
\put(165,22){\colorbox{gray}}
\put(165,27){\colorbox{gray}}
\put(169,18){\colorbox{gray}}
\put(169,22){\colorbox{gray}}
\put(169,27){\colorbox{gray}}
\put(174,18){\colorbox{gray}}
\put(174,22){\colorbox{gray}}
\put(174,27){\colorbox{gray}}

\put(180,18){\colorbox{gray}}
\put(180,22){\colorbox{gray}}
\put(180,27){\colorbox{gray}}
\put(184,18){\colorbox{gray}}
\put(184,22){\colorbox{gray}}
\put(184,27){\colorbox{gray}}
\put(189,18){\colorbox{gray}}
\put(189,22){\colorbox{gray}}
\put(189,27){\colorbox{gray}}

%
%
%
\put(165,3){\colorbox{gray}}
\put(165,7){\colorbox{gray}}
\put(165,12){\colorbox{gray}}
\put(169,3){\colorbox{gray}}
\put(169,7){\colorbox{gray}}
\put(169,12){\colorbox{gray}}
\put(174,3){\colorbox{gray}}
\put(174,7){\colorbox{gray}}
\put(174,12){\colorbox{gray}}

\put(180,3){\colorbox{gray}}
\put(180,7){\colorbox{gray}}
\put(180,12){\colorbox{gray}}
\put(184,3){\colorbox{gray}}
\put(184,7){\colorbox{gray}}
\put(184,12){\colorbox{gray}}
\put(189,3){\colorbox{gray}}
\put(189,7){\colorbox{gray}}
\put(189,12){\colorbox{gray}}

\put(120,0){\framebox(15,15)}
\put(150,0){\framebox(15,15)}
\put(165,0){\framebox(15,15)}
\put(180,0){\framebox(15,15)}
\put(120,15){\framebox(15,15)}
\put(150,15){\framebox(15,15)}
\put(165,15){\framebox(15,15)}
\put(180,15){\framebox(15,15)}
\put(120,30){\framebox(15,15)}
\put(150,30){\framebox(15,15)}
\put(165,30){\framebox(15,15)}
\put(180,30){\framebox(15,15)}
\put(120,60){\framebox(15,15)}
\put(150,60){\framebox(15,15)}
\put(165,60){\framebox(15,15)}
\put(180,60){\framebox(15,15)}

%
%
%
%
%
\put(240,48){\colorbox{gray}}
\put(240,52){\colorbox{gray}}
\put(240,57){\colorbox{gray}}
\put(244,48){\colorbox{gray}}
\put(244,52){\colorbox{gray}}
\put(244,57){\colorbox{gray}}
\put(249,48){\colorbox{gray}}
\put(249,52){\colorbox{gray}}
\put(249,57){\colorbox{gray}}

\put(255,48){\colorbox{gray}}
\put(255,52){\colorbox{gray}}
\put(255,57){\colorbox{gray}}
\put(259,48){\colorbox{gray}}
\put(259,52){\colorbox{gray}}
\put(259,57){\colorbox{gray}}
\put(264,48){\colorbox{gray}}
\put(264,52){\colorbox{gray}}
\put(264,57){\colorbox{gray}}

\put(270,48){\colorbox{gray}}
\put(270,52){\colorbox{gray}}
\put(270,57){\colorbox{gray}}
\put(274,48){\colorbox{gray}}
\put(274,52){\colorbox{gray}}
\put(274,57){\colorbox{gray}}
\put(279,48){\colorbox{gray}}
\put(279,52){\colorbox{gray}}
\put(279,57){\colorbox{gray}}

\put(285,48){\colorbox{gray}}
\put(285,52){\colorbox{gray}}
\put(285,57){\colorbox{gray}}
\put(289,48){\colorbox{gray}}
\put(289,52){\colorbox{gray}}
\put(289,57){\colorbox{gray}}
\put(294,48){\colorbox{gray}}
\put(294,52){\colorbox{gray}}
\put(294,57){\colorbox{gray}}


\put(240,33){\colorbox{gray}}
\put(240,37){\colorbox{gray}}
\put(240,42){\colorbox{gray}}
\put(244,33){\colorbox{gray}}
\put(244,37){\colorbox{gray}}
\put(244,42){\colorbox{gray}}
\put(249,33){\colorbox{gray}}
\put(249,37){\colorbox{gray}}
\put(249,42){\colorbox{gray}}

\put(255,33){\colorbox{gray}}
\put(255,37){\colorbox{gray}}
\put(255,42){\colorbox{gray}}
\put(259,33){\colorbox{gray}}
\put(259,37){\colorbox{gray}}
\put(259,42){\colorbox{gray}}
\put(264,33){\colorbox{gray}}
\put(264,37){\colorbox{gray}}
\put(264,42){\colorbox{gray}}

\put(270,33){\colorbox{gray}}
\put(270,37){\colorbox{gray}}
\put(270,42){\colorbox{gray}}
\put(274,33){\colorbox{gray}}
\put(274,37){\colorbox{gray}}
\put(274,42){\colorbox{gray}}
\put(279,33){\colorbox{gray}}
\put(279,37){\colorbox{gray}}
\put(279,42){\colorbox{gray}}

\put(285,33){\colorbox{gray}}
\put(285,37){\colorbox{gray}}
\put(285,42){\colorbox{gray}}
\put(289,33){\colorbox{gray}}
\put(289,37){\colorbox{gray}}
\put(289,42){\colorbox{gray}}
\put(294,33){\colorbox{gray}}
\put(294,37){\colorbox{gray}}
\put(294,42){\colorbox{gray}}



\put(255,18){\colorbox{gray}}
\put(255,22){\colorbox{gray}}
\put(255,27){\colorbox{gray}}
\put(259,18){\colorbox{gray}}
\put(259,22){\colorbox{gray}}
\put(259,27){\colorbox{gray}}
\put(264,18){\colorbox{gray}}
\put(264,22){\colorbox{gray}}
\put(264,27){\colorbox{gray}}

\put(270,18){\colorbox{gray}}
\put(270,22){\colorbox{gray}}
\put(270,27){\colorbox{gray}}
\put(274,18){\colorbox{gray}}
\put(274,22){\colorbox{gray}}
\put(274,27){\colorbox{gray}}
\put(279,18){\colorbox{gray}}
\put(279,22){\colorbox{gray}}
\put(279,27){\colorbox{gray}}

\put(285,18){\colorbox{gray}}
\put(285,22){\colorbox{gray}}
\put(285,27){\colorbox{gray}}
\put(289,18){\colorbox{gray}}
\put(289,22){\colorbox{gray}}
\put(289,27){\colorbox{gray}}
\put(294,18){\colorbox{gray}}
\put(294,22){\colorbox{gray}}
\put(294,27){\colorbox{gray}}


%
%
\put(270,3){\colorbox{gray}}
\put(270,7){\colorbox{gray}}
\put(270,12){\colorbox{gray}}
\put(274,3){\colorbox{gray}}
\put(274,7){\colorbox{gray}}
\put(274,12){\colorbox{gray}}
\put(279,3){\colorbox{gray}}
\put(279,7){\colorbox{gray}}
\put(279,12){\colorbox{gray}}

\put(285,3){\colorbox{gray}}
\put(285,7){\colorbox{gray}}
\put(285,12){\colorbox{gray}}
\put(289,3){\colorbox{gray}}
\put(289,7){\colorbox{gray}}
\put(289,12){\colorbox{gray}}
\put(294,3){\colorbox{gray}}
\put(294,7){\colorbox{gray}}
\put(294,12){\colorbox{gray}}


\put(240,0){\framebox(15,15)}
\put(255,0){\framebox(15,15)}
\put(270,0){\framebox(15,15)}
\put(285,0){\framebox(15,15)}
\put(240,15){\framebox(15,15)}
\put(255,15){\framebox(15,15)}
\put(270,15){\framebox(15,15)}
\put(285,15){\framebox(15,15)}
\put(240,30){\framebox(15,15)}
\put(255,30){\framebox(15,15)}
\put(270,30){\framebox(15,15)}
\put(285,30){\framebox(15,15)}
\put(240,45){\framebox(15,15)}
\put(255,45){\framebox(15,15)}
\put(270,45){\framebox(15,15)}
\put(285,45){\framebox(15,15)}

\put(-60,50){$i$-th row $\rightarrow$}
\put(20,80){$\downarrow$}
\put(15,90){$i$-th column}

\put(93,35){$\leadsto$}
\put(80,25){remove}

\put(213,35){$\leadsto$}

\put(35,-15){$h$}
\put(265,-15){$h^i$}
\end{picture}
\end{center}
\caption{The configuration corresponding to $h^i$.}
\label{pic:h^i}
\end{figure}

Let $f$ be an arbitrary element of $H^*_{\Tn}(\Xh)$. We show that $f$ can be written as a polynomial in the variables $\x_k, \y_k, t_k$ ($1\leq k \leq n$) as follows.\vspace{5pt}\\ 
{\bf Step~1.} 
We prove that $f$ can be written as a polynomial $F_1$ in the variables $\x_k, \y_k, t_k$ ($1\leq k \leq n$) on $\Sn^1$.
In particular, $f-F_1$ is equal to $0$ on $\Sn^1$, so we may assume that $f=0$ on $\Sn^1$. \\
{\bf Step~2.} Assume that $f=0$ on $\bigcup_{p=1}^{q-1} \Sn^p$ for some $q\geq 2$. 
We prove that there is a polynomial $F_q$ in the variables $\x_k, \y_k, t_k$ ($1\leq k \leq n$) such that $f-F_q$ is zero on $\bigcup_{p=1}^{q} \Sn^p$. \\

\noindent
{\bf Proof of Step~1.} 
Since $\Gamma ^{1}(h)=\Gamma (h^{1})$ is the GKM graph of $\Flags(\C^{n-1})$, 
$f$ can be written as some polynomial in variables $\x_k^{(1)}, t_k^{(1)}$ ($1\leq k \leq n-1$) on $\Sn^1$ by \eqref{eq:CohomologyFlag}. 
Here, $\x_k^{(1)}$ and $t_k^{(1)}$ on $\Sn^1$ are defined as follows: we identify $\Sn^1$ with $\S_{n-1}$ by the correspondence
$$
\Sn^1 \ni w \mapsto w(2)w(3)\cdots w(n) \in \S_{n-1} \ \ ({\rm in \ one\mathchar`-line \ notation}),
$$
and $\x_k^{(1)}=(\x_k^{(1)}(w))_{w\in\Sn^1}$ and $t_k^{(1)}=(t_k^{(1)}(w))_{w\in\Sn^1}$ are defined by
\begin{align*}
\x_k^{(1)}(w)= t_{w(k+1)}, \ t_k^{(1)}(w)=t_k \ \ \ {\rm for} \ w\in \Sn^1.
\end{align*}
This means that we have
\begin{align*}
\x_k^{(1)}= x_{k+1}, \ t_k^{(1)}=t_k \ \ \ {\rm on} \ \Sn^1.
\end{align*}
Hence, $f$ can be written as some polynomial in the variables $\x_k, t_k$ ($1\leq k \leq n$) on $\Sn^1$, as desired.\\

\noindent
{\bf Proof of Step~2.} 
We divide the proof of Step~2 into the following two cases.
\begin{itemize}
 \item[Case 1] : $2 \leq q \leq h(1)$.
 \item[Case 2] : $h(1)+1 \leq q \leq n$.
\end{itemize}
\vspace{3pt}

\noindent
{\bf Case~1.}
Let $w\in \Sn^{q}$ and $(p \ q)$ be the transposition for $1 \leq p \leq q-1$. 
Then $w(p \ q)$ is an element of $\Sn^{p}$ and we have $q \leq h(p)$ because of the assumption for Case 1. 
So the vertices $w$ and $w(p \ q)$ are connected by an edge of $\Gamma(h)$ labeled by the polynomial $t_{w(p)}-t_{w(q)}=t_{w(p)}-t_n$ up to sign. 
By the assumption $f=0$ on $\bigcup_{p=1}^{q-1} \Sn^p$, $f(w)$ is divisible by $t_{w(p)}-t_n$ for any $w\in \Sn^q$, and hence there exists a polynomial $g(w)$ in the variables $t_1,\ldots,t_n$ for each $w\in \Sn^q$ such that
\begin{equation} \label{eq:step2fw}
f(w)=g(w)\displaystyle\prod_{p=1}^{q-1}(t_{w(p)}-t_n)  \ \ \ (w\in \Sn^q).
\end{equation}
Let us write
$$
g(w) = \displaystyle \sum_{r\geq 0} g_r(w)t_n^r
$$
where each $g_r(w)$ is a polynomial in the variables $t_1,\ldots,t_{n-1}$. 

\smallskip
\noindent
{\bf Claim.} Each $g_r=(g_r(w))_{w\in \Sn^q}$ satisfies the GKM condition for $\Gamma ^q(h)$. \smallskip\\
Proof of Claim. 
Let $v\in \Sn^q$ be a vertex connected to the vertex $w$ by an edge in $\Gamma ^q(h)$. 
Then we can write $v=w(j \ i)$ for some $i$ and $j$ with $i\neq q, j\neq q$ since both $w$ and $v$ are in $\Sn^q$.
Since $w(j)=v(i)$, $w(i)=v(j)$, and $w(s)=v(s)$ for $s\neq i, j$, we have
\begin{equation} \label{eq:claim}
\displaystyle \prod_{p=1}^{q-1} (t_{w(p)}-t_n) \equiv \displaystyle \prod_{p=1}^{q-1} (t_{v(p)}-t_n) \not\equiv 0 \ \ \ ({\rm mod} \ t_{w(i)}-t_{w(j)}).
\end{equation}
On the other hand, since $f$ satisfies the GKM condition for $\Gamma (h)$ (in particular, $f$ satisfies the GKM condition for $\Gamma ^q(h)$), we have 
\begin{align*}
f(w) \equiv f(v) \ \ \ ({\rm mod} \ t_{w(i)}-t_{w(j)}), 
\end{align*}
while we have
\begin{align*}
f(w)=g(w)\displaystyle\prod_{p=1}^{q-1}(t_{w(p)}-t_n), \ \ \ f(v)=g(v)\displaystyle\prod_{p=1}^{q-1}(t_{v(p)}-t_n)
\end{align*}
from \eqref{eq:step2fw}.
These together with \eqref{eq:claim} imply that
$$
g(w) \equiv g(v) \ \ \ ({\rm mod} \ t_{w(i)}-t_{w(j)}).
$$
This means that $g_r(w) \equiv g_r(v) \ \ ({\rm mod} \ t_{w(i)}-t_{w(j)})$ for each $r$, so we proved Claim. 
\smallskip

From Claim above and the inductive assumption, each $g_r=(g_r(w))_{w\in \Sn^q}$ can be expressed as a polynomial $G_r(\x^{(q)},\y^{(q)},t^{(q)})$ in the variables $\x_k^{(q)}, \y_k^{(q)}, t_k^{(q)}$ ($1 \leq k \leq n-1$).  
Here, $\x_k^{(q)}=(\x_k^{(q)}(w))_{w\in \Sn^q}$, $\y_k^{(q)}=(\y_k^{(q)}(w))_{w\in \Sn^q}$, and $t_k^{(q)}=(t_k^{(q)}(w))_{w\in \Sn^q}$ are defined by
\begin{align}
\notag
\x_k^{(q)}(w)&:= t_{w(k)} \ \ \ \ \ \ {\rm for} \ k<q, \\  
\notag
\x_k^{(q)}(w)&:= t_{w(k+1)}  \ \ \     {\rm for} \ k\geq q, \\ 
\y_k^{(q)}(w)&:=\begin{cases} \prod_{\ell=2, \ell \neq q}^{h(1)} (t_k-t_{w(\ell)}) \ \ \ & {\rm if} \ w(1)=k, \\ \label{eq:ykq}  
                                     0 & {\rm if} \ w(1)\neq k, \end{cases}\\
\notag
t_k^{(q)}(w)&:=t_k 
\end{align}
for ${w\in \Sn^q}$.
Since this means that we have
\begin{align*}
\x_k^{(q)}= x_k \ (k<q),  \ \ 
\x_k^{(q)}= x_{k+1} \ (k\geq q), \ \
t_k^{(q)}=t_k 
\qquad \text{on $\Sn^q$},
\end{align*}
we denote $G_r(\x^{(q)},\y^{(q)},t^{(q)})$ by $G_r(\x,\y^{(q)},t)$ which is a polynomial in the variables $\x_j, t_j$ ($1 \leq j \leq n$) and $\y_k^{(q)}$ ($1 \leq k \leq n-1$).
From \eqref{eq:yksquare} and (1) in Lemma~\ref{lemma:yk}, 
it follows that the polynomial $G_r(\x,\y^{(q)},t)$ is a linear combination in the variables $\y_k^{(q)}$ ($1 \leq k \leq n-1$) over $\Z[\x_1,\ldots,\x_{n},t_1,\ldots,t_{n}]$. 
In addition, for $w\in \Sn^q$, we have 
\begin{align*}
t_k-t_{w(q)}=t_{w(1)}-t_n=(\x_1-t_n)(w) \ \ \ &{\rm if} \ w(1)=k, \\
\y_k^{(q)}(w)=\y_k(w)=0                    \ \ \ &{\rm if} \ w(1) \neq k
\end{align*}
which mean by \eqref{eq:ykq} that
$$
\y_k^{(q)}(\x_1-t_n)=\y_k \ \ \ {\rm on} \ \Sn^q.
$$
Hence, from \eqref{eq:step2fw}, we can write 
\begin{equation} \label{eq:case2f}
f=u_0(\x,t)\displaystyle\prod_{p=1}^{q-1}(t_n-\x_p)+\displaystyle\sum_{k=1}^{n-1}u_k(\x,t)\y_k\displaystyle\prod_{p=2}^{q-1}(t_n-\x_p) \ \ \ {\rm on} \ \Sn^q
\end{equation}
for some polynomials $u_0,u_1,\ldots,u_{n-1}$ in the variables $\x_k, t_k$ ($1 \leq k \leq n$). It is clear that the first summand of the right-hand side in \eqref{eq:case2f} is equal to $0$ on $\bigcup_{p=1}^{q-1} \Sn^p$. 
The second summand is equal to $0$ on $\bigcup_{p=1}^{q-1} \Sn^p$ as well, because $\y_k=0$ on $\Sn^1$ for $1 \leq k \leq n-1$ by \eqref{eq:defy} and $\prod_{p=2}^{q-1}(t_n-\x_p)=0$ on $\bigcup_{p=2}^{q-1} \Sn^p$. 
Therefore, the polynomial $f$ minus the right-hand side of \eqref{eq:case2f} is equal to $0$ on $\bigcup_{p=1}^{q} \Sn^p$. 

\smallskip

\noindent
{\bf Case 2.} 
Let $w\in \Sn^{q}$. Since $q\geq h(1)+1$ in this case, the vertices $w$ and $w(1 \ q)$ are not connected by an edge of $\Gamma(h)$. However, for $2 \leq p \leq q-1$, the vertices $w$ and $w(p \ q)$ are connected by an edge of $\Gamma(h)$ since $h(p)=n$.
Hence, similarly to Case~1, we can write 
\begin{equation} \label{eq:step3fw}
f(w)=g(w)\displaystyle\prod_{p=2}^{q-1}(t_n-t_{w(p)})  \ \ \ (w\in \Sn^q)
\end{equation}
for some polynomial $g(w)$ in the variables $t_1,\ldots,t_n$ for each $w\in\Sn^q$. 
By an argument similar to the proof of Case~1, $g:=(g(w))_{w\in \Sn^q}$ is a linear combination in the variables $\y_k^{(q)}$ ($1 \leq k \leq n-1$) over $\Z[\x_1,\ldots,\x_{n},t_1,\ldots,t_{n}]$ where $\y_k^{(q)}=(\y_k^{(q)}(w))_{w\in \Sn^q}$ is defined by
\begin{align*}
\y_k^{(q)}(w):=\begin{cases} \prod_{\ell=2}^{h(1)} (t_k-t_{w(\ell)}) \ \ \ & {\rm if} \ w(1)=k,  \\
                                     0 & {\rm if} \ w(1)\neq k \end{cases}
\end{align*}
for ${w\in \Sn^q}$.
Note that since $q \geq h(1)+1$, we do not remove the factor $t_k-t_{w(\ell)}$ for $\ell=q$ in the definition of $\y_k^{(q)}$ (cf. \eqref{eq:ykq}). Thus, in this case we have
$$
y_k^{(q)}(w)=y_k(w) \ \ \ {\rm on} \ \Sn^q.
$$
So, from \eqref{eq:step3fw}, we can write 
\begin{equation} \label{eq:case3f}
f=u(\x,\y,t)\displaystyle\prod_{p=2}^{q-1}(t_n-\x_p) \ \ \ {\rm on} \ \bigcup_{p=2}^{q}\Sn^p
\end{equation}
for some polynomial $u$ in the variables $\x_j, t_j$ ($1 \leq j \leq n$) and $\y_{k}$ ($1 \leq k \leq n-1$). In fact, this equality holds on $\Sn^q$ from the argument above, and both sides of \eqref{eq:case3f} are zero on $\bigcup_{p=2}^{q-1}\Sn^p$.
Note that the left-hand side in \eqref{eq:case3f} is equal to $0$ on $\Sn^1$, but the right-hand side is not necessarily equal to $0$. 
So, $f$ minus the right-hand side may not be equal to $0$ on $\bigcup_{p=1}^{q}\Sn^p$.
However, since we have
\begin{align*}
\y_n(w)= \prod_{\ell=2}^{h(1)} (t_n-t_{w(\ell)}) \ \ \ {\rm for} \ w\in\Sn^1  
\end{align*}
by the definition \eqref{eq:defy}, it follows that
$$
\y_n= \prod_{\ell=2}^{h(1)} (t_n-\x_{\ell}) \ \ \ {\rm on} \ \Sn^1.  
$$
This together with the assumption $q \geq h(1)+1$ implies that
\begin{equation} \label{eq:case3yn}
\displaystyle\prod_{p=2}^{q-1}(t_n-\x_p)=\y_n\displaystyle\prod_{p=h(1)+1}^{q-1}(t_n-\x_p) \ \ \ {\rm on} \ \Sn^1.
\end{equation}
Furthermore, we have that
\begin{equation} \label{eq:case3yn2}
\y_n\displaystyle\prod_{p=h(1)+1}^{q-1}(t_n-\x_p)=0 \ \ \ {\rm on} \ \Sn\setminus \Sn^1
\end{equation}
since $\y_n=0$ on $\Sn\setminus \Sn^1$ by \eqref{eq:defy}.
Now, it follows from \eqref{eq:case3f} that
\begin{equation} \label{eq:case3f2}
f=u(\x,\y,t)\Big(\displaystyle\prod_{p=2}^{q-1}(t_n-\x_p)-\y_n\displaystyle\prod_{p=h(1)+1}^{q-1}(t_n-\x_p)\Big) \ \ \ {\rm on} \ \bigcup_{p=1}^{q}\Sn^p.
\end{equation} 
In fact, \eqref{eq:case3yn} implies that the right-hand side of \eqref{eq:case3f2} is equal to $0$ on $\Sn^1$, and \eqref{eq:case3yn2} implies that the second summand in the right-hand side is equal to $0$ on $\bigcup_{p=2}^{q} \Sn^p$ as well. Therefore, $f$ minus the right-hand side in \eqref{eq:case3f2} is equal to $0$ on $\bigcup_{p=1}^{q} \Sn^p$, as desired. 

Now we have proved the claims of {\bf Step 1} and {\bf Step 2}, and hence by induction on $q$, it follows that $f$ can be written as a polynomial in the variables $\x_k, \y_k, t_k$ ($1\leq k \leq n$).
\end{proof}

We now state the main theorem of this paper where we use the convention for the product symbol as in Lemma \ref{lemma:yk}.

\begin{theorem} \label{theorem:maincohomology}
If $h=(h(1),n,\ldots,n)$, then the cohomology ring of the regular semisimple Hessenberg variety $\Xh$ is given by
$$
H^{*}(\Xh) \cong \Z[X_1,\ldots,X_n,Y_1,\ldots,Y_n]/I
$$
where $\deg(X_k)=2$, $\deg(Y_k)=2(h(1)-1)$ and $I$ is the homogeneous ideal generated by the following five types of elements:
\begin{itemize}
\item[$(1)$] $Y_{k}Y_{k'}$ \ $(1 \leq k\neq k' \leq n)$ 
\item[$(2)$] $X_1Y_{k}$ \ $(1 \leq k \leq n)$ 
\item[$(3)$] $(\prod_{\ell=h(1)+1}^{n} (-X_{\ell}))Y_{k}-\prod_{\ell=2}^{n} (-X_{\ell})$ \ $(1 \leq k \leq n)$  
\item[$(4)$] $\sum_{k=1}^{n}Y_{k}-\prod_{\ell=2}^{h(1)} (X_1-X_{\ell})$ 
\item[$(5)$] the $i$-th elementary symmetric polynomial $e_i(X_1,\ldots,X_n)$  \ $(1 \leq i \leq n)$ 
\end{itemize}
\end{theorem}

\begin{proof}
Let $\cx_k$ (resp. $\cy_{k}$) be the image of $\x_k$ (resp. $\y_{k}$) under the homomorphism $H^{*}_{\Tn}(\Xh) \to H^{*}(\Xh)$.
We define a ring homomorphism 
$$
\Z[X_1,\ldots,X_n,Y_1,\ldots,Y_n] \to H^{*}(\Xh); \ \ \ X_k \mapsto \cx_k, Y_k \mapsto \cy_k.
$$
This homomorphism is surjective from Proposition~\ref{proposition:generator} and the surjectivity of the homomorphism $H^{*}_{\Tn}(\Xh) \to H^{*}(\Xh)$. 
From Lemma~\ref{lemma:yk} and \eqref{eq:CohomologyFlag2}, the five types of elements above are all sent to $0$ under this homomorphism, and hence it induces a surjective ring homomorphism
\begin{equation} \label{eq:varphi}
\varphi: M^*(h):=\Z[X_1,\ldots,X_n,Y_1,\ldots,Y_n]/I \onto H^{*}(\Xh).
\end{equation}
Recall that the Poincar\'{e} polynomial of $\Xh$ is given by \eqref{eq:Ph1} in Lemma~ \ref{lemma:PoincarePolynomailXh}. 
Let $a_d$ be the coefficient of $\q^d$ in \eqref{eq:Ph1} for each $d\geq0$.
Then, since $H^*(\Xh)$ is a free $\Z$-module, it is enough to show that $M^{2d}(h)$ is generated by $a_d$ elements as a $\Z$-module for each $d$. 

Since the relation \eqref{eq:yksquare} was derived algebraically from the four types of relations in Lemma \ref{lemma:yk}, the same computation works as well to see that
$$
Y_k^2=Y_k \displaystyle\prod_{\ell=2}^{h(1)}(-X_{\ell}) \ \ \ {\rm in} \ M^*(h). 
$$
This together with $(1)$ and $(4)$ in $I$ implies that any element of $M^*(h)$ can be written as
\begin{equation} \label{eq:Mh}
u_0(X)+\displaystyle\sum_{k=1}^{n-1}u_k(X)Y_k
\end{equation}
where $u_0(X)$ and $u_k(X)$ are polynomials in the variables $X_1,\ldots,X_n$. 

{\bf Case 1.} We first determine the form of $u_0(X)$ in \eqref{eq:Mh}. 
Note first that $u_0(X)$ can be written as a linear combination of the form 
\begin{equation} \label{eq:FlagBasis}
X_1^{i_1}X_2^{i_2}\cdots X_n^{i_n} \ \ \ (0\leq i_j \leq n-j)
\end{equation}
over $\Z$. In fact, for each $1\leq j\leq n$, we have
\begin{equation} \label{eq:FlagBasis2}
X_j^{n-j+1}=-\displaystyle\sum_{i_1+i_2+\cdots+i_j=n-j+1 \atop i_j \leq n-j} X_1^{i_1}X_2^{i_2}\cdots X_j^{i_j} \ \ \ \ \ \ {\rm in} \ M^*(h) 
\end{equation} 
since it is known that the same relations hold in $\Z[X_1,\ldots,X_n]/(e_i(X) \mid 1\leq i \leq n)$ (cf. \cite[p163 Proof of Proposition~3]{fult97}) and we have $e_i(X_1,\ldots,X_n)$ in $I$ for all $i$. So we can use this to see that $u_0(X)$ can be written as above.
On the other hand, from $(2)$ and $(4)$ in $I$ we have 
\begin{equation} \label{eq:RelationX}
X_1\displaystyle\prod_{\ell=2}^{h(1)}(X_1-X_{\ell})=0 \ \ \ {\rm in} \ M^*(h).
\end{equation} 
This implies that $\prod_{\ell=1}^{h(1)} X_{\ell}$ can be expressed by a $\Z$-linear combination of monomials which are less than $\prod_{\ell=1}^{h(1)} X_{\ell}$ with respect to the reverse lexicographic order on the set of monomials, i.e. a monomial $X_1^{i_1}\cdots X_{n-1}^{i_{n-1}}X_n^{i_n}$ is less than or equal to a monomial $X_1^{j_1}\cdots X_{n-1}^{j_{n-1}}X_n^{j_n}$ if and only if $(i_n,i_{n-1},\ldots,i_1) \leq (j_n,j_{n-1},\ldots,j_1)$ where $\leq$ is the lexicographic order.
From this and \eqref{eq:FlagBasis2} again, we see that $u_0(X)$ can be written as a linear combination of the form \eqref{eq:FlagBasis} which does not contain the factor $\prod_{\ell=1}^{h(1)} X_{\ell}$. 

The Hilbert series of a graded $\Z$-submodule of the polynomial ring $\Z[X_1,\ldots,X_n]$ generated by the monomials of the form \eqref{eq:FlagBasis} which contains the factor $\prod_{\ell=1}^{h(1)} X_{\ell}$ is equal to
$$
\Big(\displaystyle\prod_{j=1}^{n-h(1)} \frac{1-\q^j}{1-\q}\Big)\q^{h(1)}\displaystyle\prod_{j=n-h(1)+1}^{n} \frac{1-\q^{j-1}}{1-\q}=\frac{\q^{h(1)}-\q^n}{1-\q}\displaystyle\prod_{j=1}^{n-1} \frac{1-\q^j}{1-\q}
$$
where $\deg(\q)=2$. 
Therefore, the Hilbert series of a graded $\Z$-submodule of the polynomial ring $\Z[X_1,\ldots,X_n]$ generated by the monomials of the form \eqref{eq:FlagBasis} which does not contain the factor $\prod_{\ell=1}^{h(1)} X_{\ell}$ is equal to
$$
\displaystyle\prod_{j=1}^{n} \frac{1-\q^j}{1-\q}-\frac{\q^{h(1)}-\q^n}{1-\q}\displaystyle\prod_{j=1}^{n-1} \frac{1-\q^j}{1-\q}=\frac{1-\q^{h(1)}}{1-\q}\displaystyle\prod_{j=1}^{n-1} \frac{1-\q^j}{1-\q}
$$
which coincides with the first summand of the arithmetic formula \eqref{eq:Ph1} for the Poincar\'{e} polynomial of $\Xh$.

\smallskip

{\bf Case 2.} We next determine the form of $\sum_{k=1}^{n-1}u_k(X)Y_k$ in \eqref{eq:Mh}. 
From $(2)$ in $I$, we may assume that each $u_k(X)$ is a polynomial in the variables $X_2,\ldots,X_n$. Moreover, from $(5)$ and $(3)$ in $I$, we see that there are no summand of the form $e_i(X_2,\ldots,X_n)$ $(1 \leq i \leq n-1)$ and $\prod_{\ell=h(1)+1}^{n} X_{\ell}$ in $u_k(X)$ for any $k$. This situation is the same as in Case 1 where the number of variables has decreased to $n-1$ and subscripts are reversed.
Recalling that $\deg(Y_k)=2(h(1)-1)$, the Hilbert series of a graded $\Z$-submodule of the polynomial ring $\Z[X_1,\ldots,X_n,Y_1,\ldots,Y_n]$ generated by the monomials of the form $X_n^{\ell_1}X_{n-1}^{\ell_{2}}\cdots X_2^{\ell_{n-1}}Y_k$ $(0\leq \ell_j \leq (n-1)-j, \ 1 \leq k \leq n-1)$ which does not contain the factor $\prod_{\ell=h(1)+1}^{n} X_{\ell}$ is equal to
$$
(n-1)\q^{h(1)-1}\frac{1-\q^{n-h(1)}}{1-\q}\displaystyle\prod_{j=1}^{n-2} \frac{1-\q^j}{1-\q}
$$
which coincides with the second summand in \eqref{eq:Ph1}.

Combining the conclusions of Case 1 and Case 2, we see that $M^{2d}(h)$ is generated by $a_d$ elements as a $\Z$-module for each $d$, as desired. Therefore, the homomorphism \eqref{eq:varphi} is an isomorphism.
\end{proof}

\begin{remark}
Taking $h(1)=n$ in Theorem~$\ref{theorem:maincohomology}$, $\Xh$ is the ambient flag variety $\Flags(\C^n)$ itself. Then $(3)$ in $I$ means that we have the following elements in $I$:
$$
Y_k-\prod_{\ell=2}^{n}(-X_{\ell}) \ \ \ (1 \leq k \leq n).
$$
Thus, our presentation of the cohomology ring $H^*(\Xh)$ in Theorem $\ref{theorem:maincohomology}$ reduces to the presentation \eqref{eq:CohomologyFlag2}.
\end{remark}

\begin{remark}
Recall that $\cx_k$ (resp. $\cy_{k}$) are the image of $\x_k$ (resp. $\y_{k}$) under the homomorphism $H^{*}_{\Tn}(\Xh) \to H^{*}(\Xh)$.
From the proof of Theorem~$\ref{theorem:maincohomology}$, the following two types of the monomials 
\begin{align}
\cx_1^{i_1}\cx_2^{i_2}\cdots \cx_n^{i_n} \ \ \ & {\rm which \ does \ not \ contain \ the \ factor} \ \prod_{\ell=1}^{h(1)} \cx_{\ell} \label{eq:basis1} \\
\cx_n^{\ell_1}\cx_{n-1}^{\ell_2}\cdots \cx_2^{\ell_{n-1}}\cy_k \ \ \ & {\rm which \ does \ not \ contain \ the \ factor} \ \prod_{\ell=h(1)+1}^{n} \cx_{\ell} \label{eq:basis2} 
\end{align}
running over all 
\begin{align*}
0 \leq i_j \leq n-j \quad\qquad\qquad\ \ \ &{\rm in} \ \eqref{eq:basis1} \\
0 \leq \ell_j \leq n-1-j \ \text{ and } \ 1 \leq k \leq n-1 \ \ \ &{\rm in} \ \eqref{eq:basis2}
\end{align*}
form a $\Z$-basis of $H^*(\Xh)$ when $h=(h(1),n,\ldots,n)$.
\end{remark}

\subsection{The $\Sn$-action on $H^*(\Xh)$}\label{subsec:Sn-rep}
Let $M_1$ and $M_2$ be submodules generated by the monomials of the form \eqref{eq:basis1} and \eqref{eq:basis2}, respectively.
Let us denote by $P(M_i,\q)$ the Hilbert series of $M_i$ for $i=1,2$. Then we have
\begin{equation}
\begin{split}\label{eq:P12}
P(M_1,\q)&=\frac{1-\q^{h(1)}}{1-\q}\displaystyle\prod_{j=1}^{n-1} \frac{1-\q^j}{1-\q}, \\
P(M_2,\q)&=(n-1)\q^{h(1)-1}\frac{1-\q^{n-h(1)}}{1-\q}\displaystyle\prod_{j=1}^{n-2} \frac{1-\q^j}{1-\q} 
\end{split}
\end{equation}
from the proof of Theorem~\ref{theorem:maincohomology}.
On the other hand, since we have $v \cdot \x_k=\x_k$ and $\y_k=\y_{v(k)}$ for $v\in \Sn$ by \eqref{eq:defx}, \eqref{eq:defy}, and \eqref{eq:Tymoczko}, it follows that
\begin{align} \label{eq:Repxy}
v \cdot \cx_k=\cx_k, \ \ v \cdot \cy_k=\cy_{v(k)} \ \ {\rm for \ any} \ v\in \Sn
\end{align}
as well.
Namely, $\cx_k$ is invariant and $y_k$ are covariant under the $\Sn$-action.
In particular, the $\Sn$-action on $H^*(\Xh)$ preserves the submodules $M_1$ and $M_2$.

From now on, we work with $H^*(\Xh;\C)$ in $\C$-coefficients, and we describe $M_1$ and $M_2$ as $\Sn$-representations\footnote[3]{Strictly speaking, we mean $M_1\otimes\C$ and $M_2\otimes\C$.}.
For a graded $\C$-vector space $V=\bigoplus_{i=0}^{m} V_{2i}$ such that each $V_{2i}$ is an $\Sn$-representation, we put 
$$
R(V,\q):=\sum_{i=0}^{\dimX} V_{2i} \q^i \in R(\Sn)[\q]
$$
where $R(\Sn)$ is the representation ring of $\Sn$. 
By abuse of notation, we denote
$$
R(\Xh,\q):=\sum_{i=0}^{\dimX} H^{2i}(\Xh;\C) \q^i \in R(\Sn)[\q], \ \ \ \dimX=\dim_{\C}\Xh.
$$
By an argument similar to that used to deduce \eqref{eq:P12}, we have
\begin{equation*}
\begin{split}
R(M_1,\q)&=S^{(n)}\frac{1-\q^{h(1)}}{1-\q}\displaystyle\prod_{j=1}^{n-1} \frac{1-\q^j}{1-\q},  \\
R(M_2,\q)&=S^{(n-1,1)}\q^{h(1)-1}\frac{1-\q^{n-h(1)}}{1-\q}\displaystyle\prod_{j=1}^{n-2} \frac{1-\q^j}{1-\q}
\end{split}
\end{equation*}
from 
\eqref{eq:Repxy}.
Here, for a partition $\lambda$ of $n$, $S^{\lambda}$ is the irreducible representation of $\Sn$ corresponding to $\lambda$.
Therefore, we have
\begin{align*} 
R(\Xh,\q)=S^{(n)}\frac{1-\q^{h(1)}}{1-\q}\displaystyle\prod_{j=1}^{n-1} \frac{1-\q^j}{1-\q}+ S^{(n-1,1)}\q^{h(1)-1}\frac{1-\q^{n-h(1)}}{1-\q}\displaystyle\prod_{j=1}^{n-2} \frac{1-\q^j}{1-\q}. 
\end{align*}
Equivalently, since $S^{(n)}=\Ind_{\Sn}^{\Sn} 1$ and $S^{(n-1,1)}=\Ind_{\S_{n-1}\times \S_1}^{\Sn} 1-\Ind_{\Sn}^{\Sn} 1$ in the representation ring $R(\Sn)$ where each $1$ denotes the trivial representation, we obtain the following corollary.

\begin{corollary}\label{corollary:main1}
If $h=(h(1),n,\ldots,n)$, then we have
\begin{align*} 
R(\Xh,\q)=&(\Ind_{\Sn}^{\Sn} 1)\frac{1-\q^{h(1)-1}}{1-\q}\frac{1-\q^{n}}{1-\q}\displaystyle\prod_{j=1}^{n-2} \frac{1-\q^j}{1-\q} \\
&+(\Ind_{\S_{n-1}\times \S_1}^{\Sn} 1)\q^{h(1)-1}\frac{1-\q^{n-h(1)}}{1-\q}\displaystyle\prod_{j=1}^{n-2}\frac{1-\q^j}{1-\q}. 
\end{align*}
In particular, Conjecture \ref{conjecture:h-posi} is true for $h=(h(1),n,\ldots,n)$.
\end{corollary}

We remark that Corollary \ref{corollary:main1} itself is already known from \cite[Proposition~8.1]{sh-wa14} since the Shareshian-Wachs conjecture (\cite[Conjecture 1.2]{sh-wa11}) has been solved (\cite{br-ch}, \cite{gu}). In fact, a similar formula is known for $h=(h(1),h(2),n,\ldots,n)$.
However, our result shows that it is naturally seen from our description of the ring structure of $H^*(\Xh)$.

\subsection{The invariant subring $H^*(\Xh)^{\Sn}$}
\label{subsec:inv-subring}
From our presentation of the cohomology ring $H^*(\Xh)$ and the description of the $\Sn$-action on $H^*(\Xh)$ given in Section \ref{subsec:Sn-rep}, we now give a ring presentation of the invariant subring $H^*(\Xh)^{\Sn}$.

\begin{corollary}\label{corollary:main2}
If $h=(h(1),n,\ldots,n)$, then the $\Sn$-invariant subring $H^*(\Xh)^{\Sn}$ is given by
\begin{align*}
H^*(\Xh)^{\Sn} \cong \Z[X_1,\ldots,X_n]/(e_i(X_1,\ldots,X_n), X_1\prod_{\ell=2}^{h(1)} (X_1-X_{\ell}) \mid 1 \leq i \leq n)
\end{align*}
where $e_i(X_1,\ldots,X_n)$ is the $i$-th elementary symmetric polynomial. 
\end{corollary}

\begin{proof}
Observe that any element $\alpha\in H^*(\Xh)^{\Sn} (\subset H^*(\Xh))$ can be expressed as 
$$
\sum a_{i_1,\ldots,i_n}\cx_1^{i_1}\cx_2^{i_2}\cdots \cx_n^{i_n}+ \sum b_{{\ell}_1,\ldots,{\ell}_{n-1}}\cx_n^{\ell_1}\cx_{n-1}^{\ell_2}\cdots \cx_2^{\ell_{n-1}}(\cy_1+\cdots+\cy_n) 
$$
for some integers $a_{i_1,\ldots,i_n}$ and $b_{{\ell}_1,\ldots,{\ell}_{n-1}}$
from \eqref{eq:basis1}, \eqref{eq:basis2}, and \eqref{eq:Repxy}. Since we have
\begin{equation*} 
\cy_1+\cdots+\cy_n=\prod_{\ell=2}^{h(1)} (\cx_1-\cx_{\ell}),
\end{equation*}
from Lemma~\ref{lemma:yk}~(4), $\alpha$ can be written by a polynomial in the variables $\cx_1,\ldots,\cx_n$. 
This means that the ring homomorphism 
$$
\Z[X_1,\ldots,X_n] \to H^{*}(\Xh)^{\Sn}; \ \ \ X_k \mapsto \cx_k
$$
is surjective. 
From \eqref{eq:CohomologyFlag2} we have
$$
e_i(\cx_1,\ldots,\cx_n)=0 \ \ \ {\rm in} \ H^*(\Flags(\C^n))
$$
for $i=1,\dots,n$.
Also, from \eqref{eq:RelationX}, we have
$$
\cx_1\prod_{\ell=2}^{h(1)} (\cx_1-\cx_{\ell})=0 \ \ \ {\rm in} \ H^*(\Xh)^{\Sn}.
$$
Hence, we obtain the surjective ring homomorphism 
$$
\psi: 
\Z[X_1,\ldots,X_n]/(e_i(X_1,\ldots,X_n), X_1\prod_{\ell=2}^{h(1)} (X_1-X_{\ell}) \mid 1 \leq i \leq n) \onto H^{*}(\Xh)^{\Sn}
$$
by sending $X_k$ to $\cx_k$ for any $k=1,\ldots,n$.
Let us denote by $N^*(h)$ the ring appearing in the domain of this map $\psi$.

We know that the invariant subring $H^*(\Xh)^{\Sn}$ is torsion free, since so is $H^*(\Xh)$.
Also, from \eqref{eq:P12}, the Hilbert series of $H^*(\Xh)^{\Sn}$ is 
\begin{equation*} 
P(H^*(\Xh)^{\Sn},\q)=\frac{1-\q^{h(1)}}{1-\q}\displaystyle\prod_{j=1}^{n-1} \frac{1-\q^j}{1-\q}.
\end{equation*}
On the other hand, an argument similar to that used in the proof of Theorem~\ref{theorem:maincohomology} shows that $N^*(h)$ is a free $\Z$-module whose Hilbert series 
coincides with this since we can construct a similar basis without $Y_k$. Thus, the map $\psi$ above is a surjective homomorphism between the free $\Z$-modules of the same rank, and hence it must be an isomorphism.
\end{proof}

\begin{remark}
In $\Q$-coefficients, the invariant subring
$H^*(\Xh;\Q)^{\Sn}$ is isomorphic to 
the cohomology ring of the corresponding regular nilpotent Hessenberg variety $($\cite[Theorem~B]{AHHM}$)$, and the latter is presented as 
the quotient ring of a polynomial ring $\Q[X_1,\ldots,X_n]$ by an ideal generated by the following polynomials $($\cite[Theorem~A]{AHHM}$):$
$$
\cf_{h(j),j}:=\sum_{k=1}^{j} \Big(X_k \prod_{\ell=j+1}^{h(j)} (X_k-X_{\ell}) \Big) \ \ \ (1 \leq j \leq n).
$$
We can see that our presentation of $H^*(\Xh;\Q)^{\Sn}$ is compatible with these results as follows.
Note first that $\cf_{h(1),1}=X_1 \prod_{\ell=2}^{h(1)} (X_1-X_{\ell})$ is a generator of the ideal in the right-hand side of our presentation in Corollary $\ref{corollary:main2}$.
From \cite[Remark~3.4]{AHHM}, we have the following equality of ideals of the polynomial ring $\Q[X_1,\dots,X_n]:$
$$
(\cf_{n,1},\cf_{n,2},\ldots,\cf_{n,n})=(e_1(X_1,\ldots,X_n),\ldots,e_n(X_1,\ldots,X_n)).
$$
We also know from \cite[Lemma~4.1]{AHHM} that 
$$
\cf_{n,1}\in (\cf_{h(1),1},\cf_{n,2},\ldots,\cf_{n,n})
$$
since $n\geq h(1)$. 
Thus, it follows that the ideal appearing in Corollary \ref{corollary:main2} can be written as
\begin{equation*}
(e_1(X_1,\ldots,X_n),\ldots,e_n(X_1,\ldots,X_n),\cf_{h(1),1}) = (\cf_{h(1),1},\cf_{n,2},\ldots,\cf_{n,n})
\end{equation*}
in $\Q[X_1,\dots,X_n]$ where the right-hand side is the ideal appearing in \cite[Theorem~A]{AHHM}.
\end{remark}


\end{document}